\documentclass[12pt, reqno]{amsart}
\usepackage{amsmath,amssymb, amsfonts,amscd,psfrag,graphicx,stmaryrd}
\usepackage{epsfig}
\usepackage{amsthm}
\usepackage{fullpage}
\usepackage{verbatim}
\usepackage{listings}
\usepackage[table]{xcolor} 
\usepackage{mathtools}
\DeclarePairedDelimiter{\ceil}{\lceil}{\rceil}
\DeclarePairedDelimiter\floor{\lfloor}{\rfloor}
\newcommand{\lebn}

\usepackage[misc]{ifsym}
\usepackage{caption}
\usepackage{float}
\usepackage{amsmath}
\usepackage{subfigure}

\theoremstyle{plain}
\newtheorem{prop}[equation]{Proposition}

\newtheorem{thm}[equation]{Theorem}
\newtheorem{conj}[equation]{Conjecture}

\newtheorem{cor}[equation]{Corollary}
\newtheorem{lem}[equation]{Lemma}
\newtheorem{obs}[equation]{Observation}

\theoremstyle{definition}

\newtheorem{rem}[equation]{Remark}

\numberwithin{equation}{section}

\newcommand{\D}{\Delta}

\usepackage[a4paper,left=1.5cm,right=1.5cm,top=2cm,bottom=2cm]{geometry}

\allowdisplaybreaks

\usepackage{tikz}
\usepackage{verbatim}
\usetikzlibrary{shapes,shadows,calc}
\usepgflibrary{arrows}
\usetikzlibrary{arrows, decorations.markings, calc, fadings, decorations.pathreplacing, patterns, decorations.pathmorphing, positioning}
\tikzset{nodc/.style={circle,draw=blue!50,fill=pink!80,inner sep=1.6pt}}
\tikzset{nodr/.style={circle,draw=gray,fill=gray!30!white,inner sep=1.6pt}}
\tikzset{nodel/.style={circle,draw=black,inner sep=2.2pt}}
\tikzset{nodinvisible/.style={circle,draw=white,inner sep=2pt}}
\tikzset{nodpale/.style={circle,draw=gray,fill=gray,inner sep=1.6pt}}

\tikzset{nod1/.style={circle,draw=black,fill=black,inner sep=1pt}}
\tikzset{nod2/.style={circle,draw=black,fill=blue!75!black,inner sep=1.6pt}}
\tikzset{nod3/.style={circle,draw=black,fill=black,inner sep=1.8pt}}
\tikzset{noddiam/.style={diamond,draw=black,inner sep=2pt}}
\tikzset{nodw/.style={circle,draw=black,inner sep=1.8pt}}

\usepackage[colorlinks, urlcolor=blue, linkcolor=blue, citecolor=blue]{hyperref}

\makeatletter
 \def\@textbottom{\vskip \z@ \@plus 10pt}
 \let\@texttop\relax
\makeatother

\usepackage[utf8]{inputenc}
\usepackage[T1]{fontenc}

\begin{document}

\bibliographystyle{plain}

\title[2-distance coloring of planar graphs with girth six]{An improved bound for 2-distance coloring of planar graphs with girth six}

\author{Zakir Deniz}
\thanks{The author is supported by T\" UB\. ITAK, grant no:122F250}
\address{Department of Mathematics, Duzce University, Duzce, 81620, Turkey.}
\email{zakirdeniz@duzce.edu.tr}

\date{\today}
\thanks{}

\begin{abstract}
A vertex coloring of a graph $G$ is said to be a 2-distance coloring if any two vertices at distance at most $2$ from each other receive different colors, and the least number of colors for which $G$ admits a $2$-distance coloring is known as the $2$-distance chromatic number $\chi_2(G)$ of $G$.  When $G$ is a planar graph with girth at least $6$ and maximum degree $\Delta \geq 6$, we prove that  $\chi_2(G)\leq  \Delta +4$. This improves the best known bound for 2-distance coloring of planar graphs with girth six.
\end{abstract}
\maketitle

\section{Introduction}

All graphs in this paper are assumed to be simple, i.e., finite and undirected, with no loops or multiple edges. We refer to \cite{west} for terminology and notation not defined here. Let $G$ be a graph, we use $V(G),E(G),F(G),\D(G)$ and $g(G)$ to denote the vertex, edge and face set, the maximum degree and girth of $G$, respectively. If there is no confusion in the context, we abbreviate $\D(G),g(G)$ to $\D,g$. 
A 2-distance coloring is a vertex coloring where two vertices that are adjacent or have a common neighbour receive different colors, and the smallest number of colors for which $G$ admits a 2-distance coloring is known as the 2-distance chromatic number $\chi_2(G)$ of $G$. 

In 1977, Wegner \cite{wegner} posed the following conjecture.

\begin{conj}\label{conj:main}
For every planar graph $G$, 
\[ \chi_2(G) \leq \begin{cases} 
      7 & \text{if} \ \ \Delta=3\\
      \Delta+5 & \text{if} \ \ 4\leq \Delta\leq 7 \\
       \floor[\Big]{\frac{3\Delta}{2}}+1 & \text{if} \ \  \Delta\geq 8.
   \end{cases}
\]
\end{conj}

The conjecture is still widely open. 
Thomassen \cite{thomassen} (independently by Hartke et al. \cite{hartke}) proved the conjecture for planar graphs with $\Delta = 3$. 
In general, there are some upper bounds for $2$-distance chromatic number of planar graphs. For instance, van den Heuvel and McGuinness \cite{van-den} showed that $\chi_2(G) \leq 2\Delta + 25$, while the bound $\chi_2(G) \leq \ceil[\big]{ \frac{5\D}{3}}+78$ was proved by Molloy and Salavatipour \cite{molloy}.
Recently, Zhu and Bu \cite{zhu18} proved that $\chi_2(G) \leq 20$ when $\D \leq 5$, and $\chi_2(G) \leq 5\Delta -7$ when $\D \geq 6$. In addition, Zhu et al. \cite{zhu22} showed that $\chi_2(G) \leq 5\Delta -9$ when $7 \leq \D \geq 8$.

For  planar graphs  without short cycles, Dong and Lin \cite{dong} proved that $\chi_2(G) \leq \Delta + 8$ if $g\geq 5$. In \cite{bu-zu}, Bu and Zhu showed that $\chi_2(G) \leq \Delta + 5$ if $G$ is a planar graph with $g \geq  6$, which confirms the Conjecture \ref{conj:main} for the planar graphs with girth six.


Our focus will be on planar graphs with $\D,g\geq 6$. In such a case, we improve the bound of Bu and Zhu \cite{bu-zu} by reducing one, and the main result of the paper is the following.




\begin{thm}\label{thm:main}
If $G$ is a planar graph with  $g\geq 6$ and $\Delta\geq 6$, then $\chi_2(G) \leq \Delta + 4$. 
\end{thm}


Given a planar  graph $G$, 
we denote by $\ell(f)$ the length of a face $f$ and by $d(v)$ the degree of a vertex $v$. 

A $k$-vertex is a vertex of degree $k$. A $k^{-}$-vertex is a vertex of degree at most $k$, and a $k^{+}$-vertex is a vertex of degree at least $k$. A $k$ ($k^-$ or $k^+$)-face is defined similarly. A vertex $u\in N(v)$ is called $k$-neighbour (resp. $k^-$-neighbour, $k^+$-neighbour) of $v$ if $d(u)=k$ (resp. $d(u)\leq k$, $d(u)\geq k$). A $k(d)$-vertex is a $k$-vertex adjacent to $d$ $2$-vertices. 

For a vertex $v\in V(G)$, we use $n_i(v)$ to denote the number of $i$-vertices adjacent to $v$.
Let $v\in V(G)$, we define $D(v)=\Sigma_{v_i\in N(v)}d(v_i)$. For $u,v\in V(G)$, we denote by $d(u,v)$ the distance between $u$ and $v$. Also, we set $N_i(v)=\{u \ | \ u\in V(G), \ 1\leq d(u,v) \leq i \}$ for $i\geq 1$, and so $N_1(v)=N(v)$.

%



\section{The Proof of Theorem \ref{thm:main}} \label{sec:premECE}

\subsection{The Structure of Minimum Counterexample} \label{sub:premECE}~~\medskip

Let  $G$ be a minimal counterexample to Theorem \ref{thm:main} with respect to the number of vertices such that $g(G)\geq 6$, $\Delta(G)=\D\geq 6$ and $G$ does not admit any $2$-distance coloring with $\D+4$ colors. By minimality, $G-v$ has a 2-distance coloring with $\D+4$ colors for every $v\in V(G)$. Indeed, even if $\D(G-v)<6$ for some $v\in V(G)$, then $\chi_2(G-v)\leq \Delta(G-v)+5$ by \cite{bu-zu}, and so we have $\chi_2(G-v)\leq \Delta(G-v)+5 \leq \D+4$. 
Obviously,  $G$ is connected and $\delta(G)\geq 2$. \medskip



Our proof technique is similar to \cite{bu-zu}, where we detail and adapt some of their results.
We begin with outlining some structural properties that $G$ must carry, which will be  in use in the sequel.

\begin{prop}\label{prop:properties}
$G$ has no 
 adjacent 2-vertices.
\end{prop}
\begin{proof}
Assume that $G$ has a pair of adjacent 2-vertices $u,v$. Let $w$ and $z$ be the other neighbours of $u$ and $v$, respectively. By minimality,  $G-u$ has a 2-distance coloring $f$ with $\D+4$ colors. If $f(w)\neq f(v)$, we then extend the coloring $f$ of $G-u$ into the graph $G$, since $u$ has at most $\D+2$ forbidden colors. If $f(w)=f(v)$, then we recolor $v$ with a color $c\neq f(w)$, since $v$ has at most $\D+1$ forbidden colors. In such a way, we obtain a 2-distance coloring of $G$ with $\D+4$ colors after assigning an available color to $u$. 
\end{proof}

A 2-vertex $v$ having a $3$-neighbour has the following property as $D(v)\leq \D+3$. 

\begin{cor}\label{cor:2-vertex}
Let $v$ be a 2-vertex adjacent to a $3$-vertex in $G$. If $G$ has a 2-distance coloring $f$ with $\D+4$ colors such that $v$ is uncolored vertex, then $v$ always has an available color. 
\end{cor}

\begin{lem}\label{lem:3-vertex}~\\ \vspace*{-1em}
\begin{itemize}
\item[$(1)$] If a 3-vertex $v$ has a 2-neighbour, then $D(v)\geq \D+5$.
\item[$(2)$] If a 3-vertex $v$ has a 2-neighbour, then $v$ has no  $3$-neighbour adjacent to a $2$-vertex.
\end{itemize}
\end{lem}
\begin{proof}
$(1)$.  Let $v$ be a 3-vertex having a 2-neighbour $u$. Assume for a contradiction that $D(v)\leq \D+4$. 
By minimality, $G-u$ has a 2-distance coloring $f$ with  $\D+4$ colors. Let $w$ be the other neighbour of $u$. If $f(w)\neq f(v)$, we then extend the coloring $f$ of $G-u$ into the graph $G$, since $u$ has at most $\D+3$ forbidden colors. If $f(w)=f(v)$, then we recolor $v$ with a color $[\D+4]-f(N_2(v))$, since $v$ has at most $\D+3$ forbidden colors. In addition, we assign an available color to $u$. Consequently, $G$ has a 2-distance coloring with  $\D+4$ colors, a contradiction. \medskip

$(2)$. Let $v$ be a 3-vertex having a 2-neighbour $u$. Assume to the contrary that $v$ has a $3$-neighbour $z$ adjacent to a $2$-vertex $t$.  Consider a 2-distance coloring $f$ of $G-u$ with  $\D+4$ colors, we first decolor $t$, and apply the above argument. Thus, $G$ has a  2-distance coloring with  $\D+4$ colors such that $t$ is uncolored vertex. By Corollary \ref{cor:2-vertex}, this coloring can be extended to the whole graph $G$, a contradiction.
\end{proof}

The following is an immediate consequence of Lemma \ref{lem:3-vertex}.

\begin{cor}\label{cor:3-vertex}~\\ \vspace*{-1em}
\begin{itemize}
\item[$(1)$] A 3-vertex $v$ is adjacent to at most one 2-vertex. 
\item[$(2)$] If a 3-vertex $v$ has a 2-neighbour, then $v$ has a $5^+$-neighbour.
\item[$(3)$] If $\D\geq 7$ and a 3-vertex $v$ has a 2-neighbour, then $v$ has either a $6^+$-neighbour or two $5$-neighbours.
\end{itemize}
\end{cor}

It is worth to mention that $G$ has no $3(k)$-vertex for $k\in \{2,3\}$ by Corollary \ref{cor:3-vertex}-(1). Therefore, if a $3$-vertex $v$ is adjacent to a $2$-vertex, then $v$ must be a $3(1)$-vertex.  This fact will be used in the rest of the paper. 

\newpage

\begin{lem}\label{lem:4-vertex}~\\ \vspace*{-1em}
\begin{itemize}
\item[$(1)$] If a 4-vertex $v$ has a 2-neighbour, then $D(v)\geq \D+4$.
\item[$(2)$] If a 4-vertex $v$ has a 2-neighbour and a $3(1)$-neighbour, then $D(v)\geq \D+5$.
\end{itemize}
\end{lem}
\begin{proof}
$(1)$.  Let $v$ be a 4-vertex having a 2-neighbour $u$. Assume for a contradiction that $D(v)\leq \D+3$. 
By minimality, $G-u$ has a 2-distance coloring $f$ with  $\D+4$ colors. Let us first decolor the vertex $v$. Next, we assign an available color to $u$, as it has at most $\D+3$ forbidden colors. Finally,  we give an available color to $v$ due to $D(v)\leq \D+3$. This provides a 2-distance coloring of $G$ with  $\D+4$ colors, a contradiction. \medskip

$(2)$. Let $v$ be a 4-vertex having a 2-neighbour $u$. Suppose that $v$ has a $3(1)$-neighbour $z$. Let $t$ be the $2$-neighbour of $z$. Assume to the contrary that $D(v)\leq \D+4$. Consider a 2-distance coloring $f$ of $G-u$ with  $\D+4$ colors. We first decolor $t$, and apply the above argument. Thus, $G$ has a  2-distance coloring with  $\D+4$ colors so that $t$ is an uncolored vertex. By Corollary \ref{cor:2-vertex}, this coloring can be extended to the whole graph $G$, a contradiction.
\end{proof}

The following can be easily obtained from Lemma \ref{lem:4-vertex}. 
\begin{cor}\label{cor:4-vertex}~\\ \vspace*{-1em}
\begin{itemize}
\item[$(1)$] A 4-vertex $v$ is adjacent to at most three 2-vertices.
\item[$(2)$] If  a 4-vertex $v$ is adjacent to two 2-vertices and one $3(1)$-vertex, then $v$ has a $4^+$-neighbour. 
\item[$(3)$] If  a 4-vertex $v$ is adjacent to three 2-vertices, then $v$ has a $4^+$-neighbour. 
\item[$(4)$] If $\D\geq 7$ and a 4-vertex $v$ is adjacent to two 2-vertices, then $v$ has a $4^+$-neighbour. 
\item[$(5)$] If $\D\geq 7$ and a 4-vertex $v$ is adjacent to three 2-vertices, then $v$ has a $5^+$-neighbour. 
\end{itemize}
\end{cor}

\begin{prop}\label{prop:v-is-not-adj-to-three-3(1)}
If a 4-vertex $v$ has a 2-neighbour, then $v$ is not adjacent to three $3(1)$-vertices.
\end{prop}
\begin{proof}
Let $v$ be a 4-vertex having a 2-neighbour $u$. Assume for a contradiction that $v$ is adjacent to three $3(1)$-vertices. Let $t_1,t_2,t_3$ be 2-neighbours of those $3(1)$-vertices. Consider a 2-distance coloring $f$ of $G-u$ with  $\D+4$ colors. We first decolor $v$ and all $t_i$'s.  Next, we give an available color to $u$, as it has at most  $\D+3$ forbidden colors. Notice that $v$ has at most $8$ forbidden colors. Therefore, $G$ has a  2-distance coloring with  $\D+4$ colors so that $t_1,t_2,t_3$ are uncolored vertices. By Corollary \ref{cor:2-vertex}, this coloring can be extended to the whole graph $G$, a contradiction.
\end{proof}

For a path $uvw$ with $d(v)=2$, we say that $u$ and $v$ are \emph{weak-adjacent}. 
The neighbours of a $3(1)$-vertex different from a $2$-vertex are called \emph{star-adjacent}. \medskip

We call a  vertex $v$ a \emph{special vertex} if it satisfies one of the followings:
\begin{itemize}
\item $v$ is a $4(2)$-vertex and adjacent to a $3(1)$-vertex and a $4$-vertex (or a $5$-vertex). (Type I)
\item $v$ is a $4(3)$-vertex and adjacent to a $4$-vertex (or a $5$-vertex). (Type II)
\item $v$ is a $5(4)$-vertex and adjacent to a $3(1)$-vertex. (Type III)
\item $v$ is a $5(5)$-vertex. (Type IV)
\end{itemize}

\begin{lem}\label{lem:delta6-special-not-weak-adj}~\\ \vspace*{-1em}
\begin{itemize}
\item[$(1)$] If $\D=6$ and $v$ is a special vertex of type I or II, then $v$ is not weak-adjacent to any $5^-$-vertex.
\item[$(2)$] If $\D=6$ and $v$ is a special vertex of type III or IV, then $v$ is not weak-adjacent to any $4^-$-vertex.
\end{itemize}
\end{lem}

\begin{proof}
$(1)$  Let $v$ be a special vertex of type I or II.  Denote by $t$,  the 2-vertex adjacent to the $3(1)$-neighbour of $v$ when $v$ is of type I.  Assume for a contradiction that $v$ is weak-adjacent to a $5^-$-vertex $z$. Let $u_1$ be the common neighbour of $v$ and $z$. By minimality, $G-u_1$ has a 2-distance coloring $f$ with  $\D+4=10$ colors.  Let $u_2$ be another 2-neighbour of $v$. We first decolor $t$ (if $v$ is of type I) and $u_2$, and we later recolor $v$ with an available color, since it has at most $9$ forbidden colors. Next, we recolor $u_2$ with an available color, since it has at most $9$ forbidden colors.  We also give an available color to $u_1$, since it has at most $9$ forbidden colors. Finally, we give an available color to $t$ by Corollary \ref{cor:2-vertex}, if $v$ is of type I. Thus, we get a 2-distance coloring of $G$ with $10$ colors, a contradiction. \medskip



$(2)$  Let $v$ be a special vertex of type III or IV. Denote by $t$,  the 2-vertex adjacent to the $3(1)$-neighbour of $v$  when $v$ is of type III.  Assume for a contradiction that $v$ is weak-adjacent to a $4^-$-vertex $z$. Let $u$ be the common neighbour of $v$ and $z$. By minimality, $G-u$ has a 2-distance coloring $f$ with  $\D+4=10$ colors. We first decolor $t$ if $v$ is of type I, and we give an available color to $u$ if $f(z)\neq f(v)$. Otherwise, we recolor $v$ with an available color, since it has at most $9$ forbidden colors, and so we assign an available color to $u$. Finally, we recolor $t$ with an available color by Corollary \ref{cor:2-vertex}, if $v$ is of type I.  Therefore, $G$ has a 2-distance coloring with  $\D+4=10$ colors, a contradiction. 
%
%
%
%
\end{proof}

%
%
%
%
%

\begin{lem}\label{lem:delta7-not-weak-delta-1}
If $\D\geq 7$ and $v$ is a special vertex, then $v$ is not weak-adjacent to any $(\D-1)^-$-vertex.
\end{lem}
\begin{proof}
Let $\D\geq 7$, and let $v$ be a special vertex. By contradiction, assume that $v$ is weak-adjacent to a $(\D-1)^-$-vertex $z$. Denote by $u$, the common neighbour of $v$ and $z$. By minimality, $G-u$ has a 2-distance coloring $f$ with  $\D+4$ colors.  We distinguish the special vertex $v$ into four possible cases as follows.

Let $v$ be of type I or II. Denote by $t$, the $2$-vertex adjacent to the $3(1)$-neighbour of $v$ when $v$ is of type I.  We first decolor $t$.  If $f(z)\neq f(v)$, we then extend the coloring $f$ of $G-u$ into the graph $G$, since $u$ has at most $\D+3$ forbidden colors. If $f(z)=f(v)$, then we recolor $v$ with a color $[\D+4]-f(N_2(v))$, since $v$ has at most $\D+3\geq 10$ forbidden colors. It then follows that $u$ has an available color. Finally we give an  available color to $t$ by Corollary \ref{cor:2-vertex} when $v$ is of type I. Consequently, $G$ has a 2-distance coloring with  $\D+4$ colors, a contradiction.

Let $v$ be of type III. Denote by $t$, the $2$-vertex adjacent to the $3(1)$-neighbour of $v$. We first decolor $t$ and $v$. Next, we give an available color to $u$, since it has a most $\D+3$ forbidden colors.
Now, $v$ has an available color. Finally we give an  available color to $t$ by Corollary \ref{cor:2-vertex}. This provides a 2-distance coloring of $G$ with  $\D+4$ colors, a contradiction.

Let $v$ be of type IV. We first decolor $v$, and then we give an available color to $u$, since it has a most $\D+3$ forbidden colors. Now, $v$ has an available color due to $D(v)=10<\Delta+4$. This provides a 2-distance coloring of $G$ with  $\D+4$ colors, a contradiction.
\end{proof}


\begin{prop}\label{prop:type2-vertex5}
If $\D\geq 7$ and $v$ is a special vertex of type II  having a $5(4)$-neighbour $u$, then $u$ is not weak-adjacent to any $5^-$-vertex.
\end{prop}
\begin{proof}
Assume for a contradiction that $u$ is weak-adjacent to a $5^-$-vertex $z$. Denote by $w$, the $2$-vertex adjacent to both $u$ and $z$.  By minimality, $G-u$ has a 2-distance coloring $f$ with  $\D+4$ colors. We first decolor $w$ and $v$.  It follows that $u$ has an available color, since it has at most $\D+3$ forbidden colors. We next give an available color to $v$ because it has $\D+3$ forbidden colors.  At last, we may assign an available color to $w$ as $D(w)\leq 10$. This provides a 2-distance coloring of $G$ with  $\D+4$ colors, a contradiction.\medskip
\end{proof}

%
%
%
%
%

We now define bad and semi-bad vertices according to their degrees, all of which are also special vertices. If a $5$-vertex is weak-adjacent to five $5^+$-vertex, we call it as a \emph{bad 5-vertex} (see Figure \ref{fig:bad5-a}). If a $5$-vertex is weak-adjacent to four $5^+$-vertex and star-adjacent to a $4$-vertex (or a $5$-vertex) $v_5$, we call it a \emph{semi-bad 5-vertex} (see Figure \ref{fig:bad5-b}).   
Similarly,  if a $4$-vertex is weak-adjacent to three $6$-vertices and a $4$-vertex (or a $5$-vertex) $u_4$, we call it a \emph{bad 4-vertex} (see Figure \ref{fig:bad4-a}). If a $4$-vertex is weak-adjacent to two $6$-vertex, star-adjacent to a $5$-vertex $v_3$, and adjacent to a $4$-vertex (or a $5$-vertex) $u_4$, we call it a \emph{semi-bad 4-vertex} (see Figure \ref{fig:bad4-b}).

\begin{figure}[htb]
\centering   
\subfigure[]{
\label{fig:bad5-a}
\begin{tikzpicture}[scale=1]
\node [nodr] at (0,0) (v) [label=above: {\scriptsize $v$}] {};
\node [nod2] at (-1,0) (u1) [label=above:{\scriptsize $u_1$}] {}
	edge (v);
\node [nod2] at (-2,0) (v1) [label=above:{\scriptsize $v_1$}] {}
	edge  (u1);
\node [nod2] at (-1,1) (u2) [label=above:{\scriptsize $u_2$}] {}
	edge [] (v);		
\node [nod2] at (-1.7,1.7) (v2) [label=above:{\scriptsize $v_2$}] {}
	edge [] (u2);
\node [nod2] at (1,1) (u3)  [label=above:{\scriptsize $u_3$}] {}
	edge [] (v);
\node [nod2] at (1.7,1.7) (v3) [label=above:{\scriptsize $v_3$}]  {}
	edge [] (u3);
\node [nod2] at (1,0) (u4) [label=above:{\scriptsize $u_4$}] {}
	edge [] (v);
\node [nod2] at (2,0) (v4) [label=above:{\scriptsize $v_4$}] {}
	edge [] (u4);	
\node [nod2] at (0,-.8) (u5) [label=right:{\scriptsize $u_5$}] {}
	edge [] (v);
\node [nod2] at (0,-1.6) (v5) [label=right:{\scriptsize $v_5$}] {}
	edge [] (u5);		
\node at (1,-2.55) (asd)     {};				
\end{tikzpicture}  }\hspace*{1cm}
\subfigure[]{
\label{fig:bad5-b}
\begin{tikzpicture}[scale=1]
\node [nodr] at (0,0) (v) [label=above:{\scriptsize $v$}] {};
\node [nod2] at (-1,0) (u1) [label=above:{\scriptsize $u_1$}] {}
	edge (v);
\node [nod2] at (-2,0) (v1) [label=above:{\scriptsize $v_1$}] {}
	edge  (u1);
\node [nod2] at (-1,1) (u2) [label=above:{\scriptsize $u_2$}] {}
	edge [] (v);		
\node [nod2] at (-1.7,1.7) (v2) [label=above:{\scriptsize $v_2$}] {}
	edge [] (u2);
\node [nod2] at (1,1) (u3)  [label=above:{\scriptsize $u_3$}] {}
	edge [] (v);
\node [nod2] at (1.7,1.7) (v3) [label=above:{\scriptsize $v_3$}]  {}
	edge [] (u3);
\node [nod2] at (1,0) (u4) [label=above:{\scriptsize $u_4$}] {}
	edge [] (v);
\node [nod2] at (2,0) (v4) [label=above:{\scriptsize $v_4$}] {}
	edge [] (u4);	
\node [nod2] at (0,-.8) (u5) [label=right:{\scriptsize $u_5$}] {}
	edge [] (v);
\node [nod2] at (-1,-1.6) (z) [label=left:{\scriptsize $z$}] {}
	edge [] (u5);
\node [nod2] at (-1,-2.4) (y) [label=left:{\scriptsize $y$}] {}
	edge [] (z);
\node [nod2] at (1,-1.6) (v5) [label=left:{\scriptsize $v_5$}] {}
	edge [] (u5);	
\node [nod2] at (2,-.8) (v51) {}
	edge [dotted] (v5);	
\node [nod2] at (2,-1.4) (v52) {}
	edge [] (v5);
\node [nod2] at (2,-2) (v53) {}
	edge [] (v5);
\node [nod2] at (2,-2.6) (v54) {}
	edge [] (v5);											
\end{tikzpicture}  }
\caption{$(a)$ A bad $5$-vertex. \ $(b)$ A semi-bad $5$-vertex.}
\label{fig:bad5}
\end{figure}
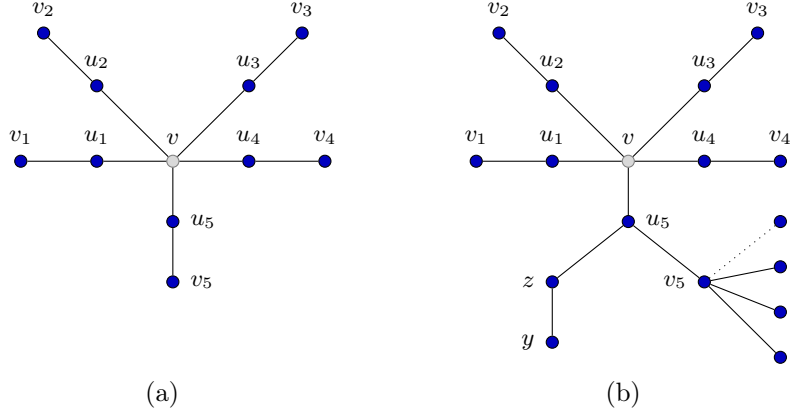

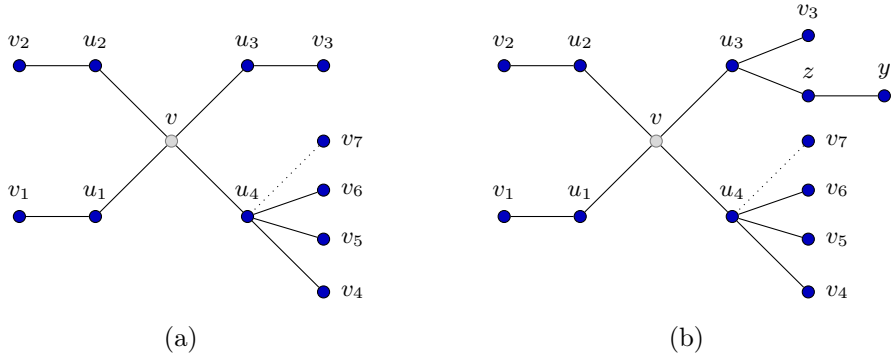
\begin{figure}[htb]
\centering   
\subfigure[]{
\label{fig:bad4-a}
\begin{tikzpicture}[scale=1]
\node [nodr] at (0,0) (v) [label=above:{\scriptsize $v$}] {};
\node [nod2] at (-1,-1) (u1) [label=above:{\scriptsize $u_1$}] {}
	edge (v);
\node [nod2] at (-2,-1) (v1) [label=above:{\scriptsize $v_1$}] {}
	edge  (u1);
\node [nod2] at (-1,1) (u2) [label=above:{\scriptsize $u_2$}] {}
	edge [] (v);		
\node [nod2] at (-2,1) (v2) [label=above:{\scriptsize $v_2$}] {}
	edge [] (u2);
\node [nod2] at (1,1) (u3)  [label=above:{\scriptsize $u_3$}] {}
	edge [] (v);
\node [nod2] at (2,1) (v3) [label=above:{\scriptsize $v_3$}]  {}
	edge [] (u3);
\node [nod2] at (1,-1) (u4) [label=above:{\scriptsize $u_4$}] {}
	edge [] (v);
\node [nod2] at (2,0) (v7) [label=right:{\scriptsize $v_7$}] {}
	edge [dotted] (u4);	
\node [nod2] at (2,-.65) (v6)  [label=right:{\scriptsize $v_6$}]  {}
	edge [] (u4);
\node [nod2] at (2,-1.3) (v5)  [label=right:{\scriptsize $v_5$}]  {}
	edge [] (u4);
\node [nod2] at (2,-2) (v4)  [label=right:{\scriptsize $v_4$}] {}
	edge [] (u4);					
\end{tikzpicture}  }\hspace*{1cm}
\subfigure[]{
\label{fig:bad4-b}
\begin{tikzpicture}[scale=1]
\node [nodr] at (0,0) (v) [label=above:{\scriptsize $v$}] {};
\node [nod2] at (-1,-1) (u1) [label=above:{\scriptsize $u_1$}] {}
	edge (v);
\node [nod2] at (-2,-1) (v1) [label=above:{\scriptsize $v_1$}] {}
	edge  (u1);
\node [nod2] at (-1,1) (u2) [label=above:{\scriptsize $u_2$}] {}
	edge [] (v);		
\node [nod2] at (-2,1) (v2) [label=above:{\scriptsize $v_2$}] {}
	edge [] (u2);
\node [nod2] at (1,1) (u3)  [label=above:{\scriptsize $u_3$}] {}
	edge [] (v);
\node [nod2] at (2,1.4) (v3) [label=above:{\scriptsize $v_3$}]  {}
	edge [] (u3);
\node [nod2] at (2,.6) (z) [label=above:{\scriptsize $z$}]  {}
	edge [] (u3);	
\node [nod2] at (3,.6) (y) [label=above:{\scriptsize $y$}]  {}
	edge [] (z);		
\node [nod2] at (1,-1) (u4) [label=above:{\scriptsize $u_4$}] {}
	edge [] (v);
\node [nod2] at (2,0) (v7) [label=right:{\scriptsize $v_7$}] {}
	edge [dotted] (u4);	
\node [nod2] at (2,-.65) (v6)  [label=right:{\scriptsize $v_6$}]  {}
	edge [] (u4);
\node [nod2] at (2,-1.3) (v5)  [label=right:{\scriptsize $v_5$}]  {}
	edge [] (u4);
\node [nod2] at (2,-2) (v4)  [label=right:{\scriptsize $v_4$}] {}
	edge [] (u4);															
\end{tikzpicture}  }
\caption{$(a)$ A bad $4$-vertex. \ $(b)$ A semi-bad $4$-vertex.}
\label{fig:bad4}
\end{figure}

The next observations follow from Lemma \ref{lem:delta6-special-not-weak-adj}.

\begin{obs}\label{obs:special-to-bad} 
If $\D=6$, then the followings hold.
\begin{itemize}
\item[$(1)$] If a special vertex of type I is star-adjacent to a $5$-vertex, then it is a semi-bad $4$-vertex. 
\item[$(2)$] If a special vertex of type III is star-adjacent to a $4$-vertex (or a $5$-vertex), then it is a semi-bad $5$-vertex. 
\item[$(3)$] Every special vertex of type II (resp. type IV) is a bad $4$-vertex (resp. a bad $5$-vertex). 
\end{itemize}
\end{obs}

\newpage

\begin{lem}\label{lem:4-vertex-color-class}
Let $\D=6$, and suppose that $G$ has a bad or a semi-bad $4$-vertex $v$ (see Figure \ref{fig:bad4}).  Then $G-v$ has a 2-distance coloring $f$ which satisfies the followings.
\begin{itemize}
\item[$(a)$] If $v$ is a semi-bad 4-vertex adjacent to a $4$-vertex, then all $f(v_i)$'s are distinct for $i\in [6]$. Also  $f(N[v_i]-\{u_i\}) \subseteq \{1,2,\ldots,6\}$ for  $i\in [3]$, and $\{1,2,\ldots,6\} \subseteq (f(N[v_4])\cup f(N[v_5])\cup f(N[v_6]))-f(u_4)  $.
\item[$(b)$] If $v$ is a semi-bad 4-vertex adjacent to a $5$-vertex, then  $f(N[v_i]-\{u_i\}) \subseteq \{1,2,\ldots,6\}$ for  $i\in [3]$, and $\{1,2,\ldots,6\} \subseteq (f(N[v_4])\cup f(N[v_5])\cup f(N[v_6]))-f(u_4)  $.
\item[$(c)$] If $v$ is a bad 4-vertex adjacent to a $4$-vertex, then all $f(v_i)$'s are distinct for $i\in [6]$.
Also  $f(N[v_i]-\{u_i\}) = \{1,2,\ldots,6\}$ for  $i\in [3]$, and $\{1,2,\ldots,6\} \subseteq (f(N[v_4])\cup f(N[v_5])\cup f(N[v_6]))-f(u_4)  $.
\item[$(d)$] If $v$ is a bad 4-vertex adjacent to a $5$-vertex, then $f(N[v_i]-\{u_i\})=\{1,2,\ldots,6\}$ for  $i\in [3]$. 
\end{itemize}
\end{lem}

\begin{proof}
$(a)$ Let $v$ be a semi-bad $4$-vertex adjacent to a $4$-vertex $u_4$ (see Figure \ref{fig:bad4-b}). 
It follows from  the definition that the vertices $v_1,v_2$ are $6$-vertices, and $v_3$ is a $5$-vertex. By minimality, $G-v$ has a 2-distance coloring $f$ with  $\D+4=10$ colors. We decolor the vertices $u_1,u_2,u_3,z$.

If there exist $i,j\in [6]$ with $i\neq j$ such that $f(v_i)=f(v_j)$, then $v$ has at most $6$ forbidden colors, we recolor $u_3,u_2,u_1,v$ in succession, respectively. At last, we recolor $z$ with an available color by Corollary \ref{cor:2-vertex}. 
Thus, we obtain a 2-distance coloring $f$ with  $\D+4=10$ colors, a contradiction.  We conclude that all $f(v_i)$'s are distinct for $i\in [6]$. Let $f(v_i)=a_i=i$ for $i\in [6]$. Consider  the 2-distance coloring $f$ of $G-v$, we decolor the vertices $u_1,u_2,u_3,z$.

Notice first that  if there exists  $a\in f(N[v_i]-\{u_i\})$ for some $i\in [2]$ such that $a$ is available color for $u_4$, then we recolor $u_4$ with $a$, and we recolor $u_{3-i},u_3,v$ with some available colors in succession. Observe that $u_i$ has at most $9$ forbidden colors, since $a\in f(N[v_i]-\{u_i\}) \cap f(u_4)$. Thus, $u_i$ has an available color by which we color it.  Lastly, we recolor $z$ with an available by Corollary \ref{cor:2-vertex}. Therefore, we obtain a 2-distance coloring $f$ with  $\D+4=10$ colors, a contradiction. We may further assume that  $f(u_4)\notin f(N[v_i]-\{u_i\})$ for each $i\in [2]$.  On the other hand, if there exists  $a\in f(N[v_i]-\{u_i\})$ such that $a\notin f(N[v_j]-\{u_j\})$ for some $i,j\in [2]$, then we recolor $u_j$ with $a$ (recall that $f(u_4)\neq a$), and we recolor $u_3,v$ with some available colors in succession. Similarly as above,  $u_i$ has an available color by which we color it.  Lastly, we recolor $z$ with an available by Corollary \ref{cor:2-vertex}. 
Thus, we obtain a 2-distance coloring $f$ with  $\D+4=10$ colors, a contradiction. So, $f(N[v_1]-\{u_1\})=f(N[v_2]-\{u_2\})$.

Let $f(N[v_i]-\{u_i\})=C_i$ for $i\in [2]$, and let $f(N[v_3]-\{u_3\})\cup f(y)=C_3$. We may assume without loss of generality that $C_1=\{1,2,\ldots,6\}$. If there exists  $a\in C_i$ such that $a\notin C_j$ for $i,j\in [3]$, then we recolor $u_j$ with $a$, and we apply the same process as above. Therefore, we obtain a 2-distance coloring $f$ with  $\D+4=10$ colors, a contradiction. Thus, we have $C_1=C_2=C_3$, and this implies that $f(N[v_i]-\{u_i\}) \subseteq \{1,2,\ldots,6\}$ for  $i\in [3]$.
On the other hand, if there exists  $a\in C_i$ for some $i\in [3]$ such that $a$ is available color for $u_4$, then we recolor $u_4$ with $a$. By a similar process as above, we recolor the remaining vertices in succession, a contradiction.   Thus, we conclude that  $\{1,2,\ldots,6\} \subseteq f((N[v_4]\cup N[v_5] \cup N[v_6]) -\{u_4\})$. \medskip

$(b)$ Let $v$ be a semi-bad $4$-vertex adjacent to a $5$-vertex $u_4$ (see Figure \ref{fig:bad4-b}). Notice that $v_1$ and $v_2$ are $6$-vertices, and $v_3$ is a $5$-vertex. By minimality, $G-v$ has a 2-distance coloring $f$ with  $\D+4=10$ colors. We decolor 2-vertices $u_1,u_2,u_3,z$.

Let $f(N[v_i]-\{u_i\})=C_i$ for $i\in [2]$, and let $f(N[v_3]-\{u_3\})\cup f(y)=C_3$. We may assume without loss of generality that $C_1=\{1,2,\ldots,6\}$.  If there exists  $a\in C_i$ for some $i\in [3]$ such that $a$ is available color for $u_4$, then we recolor $u_4$ with $a$. We also recolor $v$ with an available color, since it has at most $8$ forbidden colors. Next, we recolor $u_k$ with an available color for each $k\in \{1,2,3\}-\{i\}$, since it has at most $9$ forbidden colors. Now, observe that $u_i$ has at most $9$ forbidden colors due to $f(u_4) \in f(N[v_i]-\{u_i\})$. Finally, we recolor $z$ with an available by Corollary \ref{cor:2-vertex}. As a result, we obtain a 2-distance coloring $f$ with  $\D+4=10$ colors, a contradiction. Thus we conclude that $\{1,2,\ldots,6\} \subseteq f((N[v_4]\cup N[v_5] \cup N[v_6] \cup N[v_7]) -\{u_4\})$. This particularly implies that  $f(u_4)\notin C_i$ for each $i\in [3]$. 


If there exists  $a\in C_i$ such that $a\notin C_j$ for $i,j\in [3]$, then we recolor $u_j$ with $a$. In addition, we recolor $v$ with an available color, since it has at most $9$ forbidden colors. Next, we recolor $u_k$ with an available color for $k\in \{1,2,3\}-\{i,j\}$, since it has at most $9$ forbidden colors. Now, observe that $u_i$ has at most $9$ forbidden colors due to $f(u_j) \in f(N[v_i]-\{u_i\})$. Finally, we recolor $z$ with an available color by Corollary \ref{cor:2-vertex}. As a result, we obtain a 2-distance coloring $f$ with  $\D+4=10$ colors, a contradiction. Thus $C_1=C_2=C_3$ and so $f(N[v_i]-\{u_i\}) \subseteq \{1,2,\ldots,6\}$ for  $i\in [3]$. \medskip

$(c)$ Let $v$ be a bad $4$-vertex adjacent to a $4$-vertex $u_4$ (see Figure \ref{fig:bad4-a}). Notice that $v_1,v_2,v_3$ are $6$-vertices.   By minimality, $G-v$ has a 2-distance coloring $f$ with  $\D+4=10$ colors. We decolor the 2-vertices $u_1,u_2,u_3$. If there exist $i,j\in [6]$ with $i\neq j$ such that $f(v_i)=f(v_j)$, then $v$ has at most $6$ forbidden colors, we recolor $u_1,u_2,u_3,v$ in succession, respectively. 
Therefore, we obtain a 2-distance coloring $f$ with  $\D+4=10$ colors, a contradiction.  This implies that all $f(v_i)$'s are distinct for $i\in [6]$. Let $f(v_i)=a_i=i$ for $i\in [6]$.

Let $f(N[v_i]-\{u_i\})=C_i$ for $i\in [3]$. We may assume without loss of generality that $C_1=\{1,2,\ldots,6\}$. If there exists  $a\in C_i$ for some $i\in [3]$ such that $a$ is available color for $u_4$, then we recolor $u_4$ with $a$. By a similar process as above, we recolor the remaining vertices in succession, a contradiction.   Thus we conclude that  $\{1,2,\ldots,6\} \subseteq f((N[v_4]\cup N[v_5] \cup N[v_6]) -\{u_4\})$. This particularly implies that  $f(u_4)\notin C_i$ for $i\in [3]$. 

If there exists  $a\in C_i$ such that $a\notin C_j$ for $i\in [3]$, then we recolor $u_j$ with $a$. For $k\in \{1,2,3\}-\{i,j\}$, we next recolor $v$ and $u_k$  with some available colors,  since each one has at most $9$ forbidden colors. Now observe that $u_i$ has at most $9$ forbidden colors  because of $f(u_j) \in f(N[v_i]-\{u_i\})$. Therefore, we obtain a 2-distance coloring $f$ with  $\D+4=10$ colors, a contradiction. Thus $C_1=C_2=C_3$, and so $f(N[v_i]-\{u_i\}) \subseteq \{1,2,\ldots,6\}$ for  $i\in [3]$.\medskip

$(d)$ Let $v$ be a bad $4$-vertex adjacent to a $5$-vertex $u_4$ (see Figure \ref{fig:bad4-a}). Notice that $v_1,v_2,v_3$ are $6$-vertices. By minimality, $G-v$ has a 2-distance coloring $f$ with  $\D+4=10$ colors. We decolor the 2-vertices $u_1,u_2,u_3$.

Let $f(N[v_i]-\{u_i\})=C_i$ for $i\in [3]$. We may assume without loss of generality that $C_1=\{1,2,\ldots,6\}$. If there exists  $a\in C_i$ for some $i\in [3]$ such that $a$ is available color for $u_4$, then we recolor $u_4$ with $a$. By a similar process as above, we recolor the remaining vertices in succession, a contradiction.  Therefore we may assume that $f(u_4)\notin C_i$ for $i\in [3]$. If there exists  $a\in C_i$ such that $a\notin C_j$ for $i\in [3]$, then we recolor $u_j$ with $a$ (recall that $f(u_4)\neq a$).  We recolor $v$ with an available color, since it has at most $9$ forbidden colors. Next, we recolor $u_k$ with an available color for $k\in \{1,2,3\}-\{i,j\}$, since it has at most $9$ forbidden colors. Now observe that $u_i$ has at most $9$ forbidden colors due to $f(u_j) \in f(N[v_i]-\{u_i\})$.  Thus, we obtain a 2-distance coloring $f$ with  $\D+4=10$ colors, a contradiction. Thus $C_1=C_2=C_3$. 
%
\end{proof}

\begin{lem}\label{lem:5-vertex-color-class}
Let $\D=6$, and suppose that $G$ has a bad or a semi-bad $5$-vertex $v$ (see Figure \ref{fig:bad5}).  Then $G-v$ has a 2-distance coloring $f$ satisfying:
\begin{itemize}
\item[$(a)$] If $v$ is a bad 5-vertex and weak-adjacent to a $5$-vertex, then all $f(v_i)$'s are distinct for $i\in [5]$, and $\{1,2,\ldots,5\} \subseteq f(N[v_i]-\{u_i\})$ for  $i\in [5]$. 
\item[$(b)$] If $v$ is  a bad 5-vertex and weak-adjacent to no $5^-$-vertex, then  $\{1,2,\ldots,5\} \subseteq f(N[v_i]-\{u_i\})$ for  $i\in [5]$. 
\item[$(c)$] If $v$ is a semi-bad 5-vertex such that $v$ is weak-adjacent to a $5$-vertex or  star-adjacent to a $4$-vertex, then all $f(v_i)$'s are distinct for  $i\in [5]$ and $\{1,2,\ldots,5\} \subseteq f(N[v_i]-\{u_i\})$ for  $i\in [4]$.  In particular, $\{1,2,\ldots,5\} \subseteq f(N[v_5]-\{u_5\})\cup f(y)$. 
\end{itemize}
\end{lem}

\begin{proof}
$(a)$ Let $v$ be a bad $5$-vertex (see Figure \ref{fig:bad5-a}).  By minimality, $G-v$ has a 2-distance coloring $f$ with  $\D+4=10$ colors. We decolor 2-vertices $u_1,u_2,\ldots,u_5$. Each $v_i$ is a $5$ or $6$-vertex by the definition of $v$ together with $\D=6$. Suppose that there exists $i\in [5]$ such that $v_i$ is a $5$-vertex (say $v_1$, by symmetry). Assume for a contradiction that there exist $i,j\in [5]$ with $i\neq j$ such that $f(v_i)=f(v_j)$. Then, we recolor $u_2,u_3,u_4,u_5$ in succession, and we next consider the vertex $u_1$. Since it has at most $9$ forbidden colors due to $d(v_1)=5$, we recolor it with an available color. Now observe that $v$ has at most $9$ forbidden colors, since $f(v_i)=f(v_j)$,  and so we assign  an available color to $v$. Thus, we obtain a 2-distance coloring $f$ with  $\D+4=10$ colors, a contradiction. This implies that all $f(v_i)$'s are distinct. Let $f(v_i)=a_i=i$ for $i\in [5]$. 

If $i\notin f(N[v_j]-\{u_j\})$ for some $i,j\in [5]$ with $i\neq j$, then we recolor $u_j$ with $i$, and recolor the remaining vertices of $N(v)$ in succession. Finally, $v$ has at most $9$ forbidden colors, since    $f(u_j)=f(v_i)=i$,  so we assign  an available color to $v$.  Therefore, we obtain a 2-distance coloring $f$ with  $\D+4=10$ colors, a contradiction.  It then follows that  $\{1,2,\ldots,5\} \subseteq f(N[v_i]-\{u_i\})$ for each $i\in [5]$.\medskip


$(b)$ Let $v$ be a bad $5$-vertex (see Figure \ref{fig:bad5-a}). Suppose that $v$ is weak-adjacent to no $5^-$-vertex, i.e., all $v_i$'s are $6$-vertices.   By minimality, $G-v$ has a 2-distance coloring $f$ with  $\D+4=10$ colors. We may assume without loss of generality that $f(N[v_1]-\{u_1\})=\{1,2,\ldots,6\}$. We start with decoloring the 2-vertices $u_1,u_2,\ldots,u_5$. 

Suppose first that there exist $i,j\in [5]$ with $i\neq j$ such that $f(v_i)=f(v_j)$. In such a case, we claim that $f(N(v_k)-\{u_k\})=f(N(v_r)-\{u_r\})$ for each $k,r\in [5]$. Indeed, if $c\in f(N(v_k)-\{u_k\}) \setminus f(N(v_r)-\{u_r\})$, then  we recolor $u_r$ with the color $c$ and we recolor the remaining  vertices of $N(v)$, except for $u_k$, in succession. Since $u_k$ has at most $9$ forbidden colors due to $f(u_r)\in f(N(v_k)-\{u_k\})$, we recolor it with an available color.  Finally, $v$ has at most $9$ forbidden colors since    $f(v_i)=f(v_j)$, and so we assign an available color to it.  Thus, we obtain a 2-distance coloring $f$ with  $\D+4=10$ colors, a contradiction. Hence, we have $f(N(v_k)-\{u_k\})=f(N(v_r)-\{u_r\})$ for each $k,r\in [5]$. Particularly, we have $\{1,2,\ldots,5\} \subseteq f(N[v_i]-\{u_i\})$ for  $i\in [5]$. 

Next, we suppose that  all $f(v_i)$'s are distinct for $i\in [5]$. Let $f(v_i)=a_i=i$ for $i\in [5]$. If $i\notin f(N[v_j]-\{u_j\})$ for some $i,j\in [5]$ with $i\neq j$, then we recolor $u_j$ with $i$, and we recolor the remaining vertices of $N(v)$ in succession. Finally, $v$ has at most $9$ forbidden colors since    $f(u_j)=f(v_i)=i$, and so we assign an available color to $v$.  Consequently, we obtain a 2-distance coloring $f$ with  $\D+4=10$ colors, a contradiction.  Hence,  $\{1,2,\ldots,5\} \subseteq f(N[v_i]-\{u_i\})$ for each $i\in [5]$. \medskip

%

$(c)$  Let $v$ be a semi-bad $5$-vertex (see Figure \ref{fig:bad5-b}). By minimality, $G-v$ has a 2-distance coloring $f$ with  $\D+4=10$ colors.

Suppose first that  $v_5$ is a $4$-vertex. We decolor the vertices $u_1,u_2,\ldots,u_5,z$. Note that $u_5$ currently has at most $5$ forbidden colors. Assume for a contradiction that there exist $i,j\in [5]$ with $i\neq j$ such that $f(v_i)=f(v_j)$. We recolor $u_1,u_2,u_3,u_4$ in succession. Now, consider the vertex $u_5$, since it has at most $9$ forbidden colors, we recolor $u_5$ with an available color.  Finally, $v$ has at most $9$ forbidden colors, since    $f(v_i)=f(v_j)$,  so we assign an available color to $v$. Besides, we recolor $z$ with an available by Corollary \ref{cor:2-vertex}.  Therefore, we obtain a 2-distance coloring $f$ with  $\D+4=10$ colors, a contradiction. This means that all $f(v_i)$'s are distinct for $i\in [5]$. Let $f(v_i)=a_i=i$ for $i\in [5]$.

We next suppose that $v$ is weak-adjacent to a $5$-vertex, say $v_1$. We may also assume that  $v_5$ is a $5$-vertex, since otherwise the claim follows from above result.
Assume for a contradiction that there exist $i,j\in [5]$ with $i\neq j$ such that $f(v_i)=f(v_j)$.  We prove the claim in a similar way as above.  Let us decolor  $u_1,u_2,\ldots,u_5$ and $z$. Note that  $v$ currently has at most $4$ forbidden colors since $f(v_i)=f(v_j)$. Now, we recolor $u_2,u_3,u_4,u_5$ in succession. Consider the vertex $u_1$, we give an available color to it, since it has at most $9$ forbidden colors due to $d(v_1)=5$. Finally, $v$ has at most $9$ forbidden colors because of $f(v_i)=f(v_j)$,  so we assign an available color to  $v$. Besides, we recolor $z$ with an available by Corollary \ref{cor:2-vertex}.  Thus, we obtain a 2-distance coloring $f$ with  $\D+4=10$ colors, a contradiction. This means that all $f(v_i)$'s are distinct for $i\in [5]$. Let $f(v_i)=a_i=i$ for $i\in [5]$.

Let us now show the rest of the claim.  Let $f(N[v_i]-\{u_i\})=C_i$ for $i\in [4]$ and let $f(N(v_5)-\{u_5\})\cup f(y)=C_5$.  We decolor the vertices $u_1,u_2,\ldots,u_5,z$. If $i\notin C_j$ for some $i,j\in [5]$ with $i\neq j$, then we recolor $u_j$ with $i$, and we recolor the vertices $N(v)-\{u_i,u_j\}$ in succession. Next, we give an available color to $u_i$, since it has at most $9$ forbidden colors due to $f(u_j)=f(v_i)=i$. By the same reason, we can assign an available color to  $v$. Besides, we recolor $z$ with an available by Corollary \ref{cor:2-vertex}.  Thus, we obtain a 2-distance coloring $f$ with  $\D+4=10$ colors, a contradiction. It then follows that  $\{1,2,\ldots,5\} \subseteq f(N[v_i]-\{u_i\})$ for each $i\in [4]$, and $\{1,2,\ldots,5\} \subseteq f(N[v_5]-\{u_5\})\cup f(y)$. 
\end{proof}

\begin{lem}\label{lem:bad-5-vertices-not-adjacent}
Let $\D=6$ and let $v$ be a bad or a semi-bad $5$-vertex. If $v$ is weak-adjacent (resp. star-adjacent) to a $5$-vertex (resp. $4$-vertex) $w$, then $v$ is not weak-adjacent to a bad or a semi-bad $5$-vertex $u$ with $u\neq w$.
\end{lem}
\begin{proof}
Let $v$ be a bad $5$-vertex, and suppose that it is weak-adjacent to a $5$-vertex $w$. Assume to the contrary that $v$ is weak-adjacent to a bad $5$-vertex $u$ with $u\neq w$  as in Figure \ref{fig:adjacent5}. By Lemma \ref{lem:5-vertex-color-class}-(a), $G-v$ has a 2-distance coloring $f$ with $10$ colors such that the colors $f(v_1'),f(v_2'),f(v_3'),f(v_4'),f(u)$ are distinct, and $ \{1,2,\ldots,5\} \subseteq f(N[v_i']-\{v_i\})\cap f(N[u]-\{x\})$ for $i\in [4]$. We may assume that $f(v_i)=i$ for $i\in [4]$ and $f(u)=5$, and so $f(\{u_1,u_2,u_3,u_4\})=\{1,2,3,4\}$. In particular, $f(N[v_1']-\{v_1\})=f(N[u]-\{x\})=\{1,2,3,4,5\}$, since both $v_1'=w$ and $u$ are $5$-vertices. We then decolor $v_1,v_2,v_3,v_4,x,u$. Since $u$ currently has at most $8$ forbidden colors, it has two available colors, so we assign a color $a$ to $u$ for which $a \in \{6,7,\ldots,10\}$. Also, we assign the color $a$ to $v_1$, since all the colors in $\{6,7,\ldots,10\}$ currently are available for $v_1$. We next recolor the vertices $v_2,v_3,v_4,x$ in succession. Finally, we give an available color to $v$, since it has at most $9$ forbidden colors due to $f(u)=f(v_1)$.  Consequently, $G$ has a 2-distance coloring with $10$ colors, a contradiction.

Notice that if $v$ is star-adjacent to a $4$-vertex $w=v_1'$, i.e., $v$ is a semi-bad vertex, then we decolor the 2-vertex $z$ (see Figure \ref{fig:adjacent5-semi}) and apply the same argument as above. It follows from Corollary \ref{cor:2-vertex} that $G$ has a 2-distance coloring with $10$ colors, a contradiction.
\end{proof}

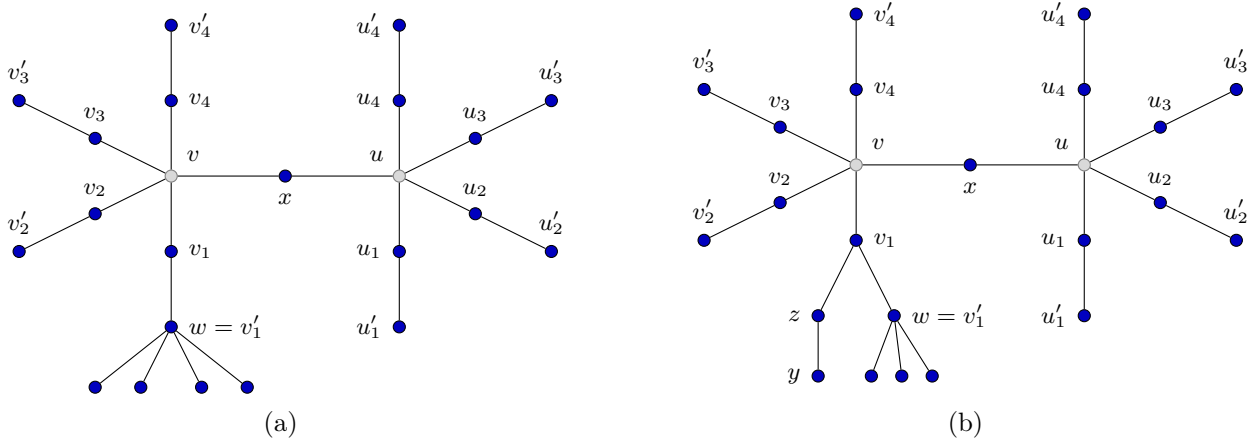
\begin{figure}[htb]
\centering   
\subfigure[]{
\label{fig:adjacent5}
\begin{tikzpicture}[scale=1]
\node [nodr] at (0,0) (v) [label=above right:{\scriptsize $v$}] {};
\node [nod2] at (0,-1) (v1) [label=right:{\scriptsize $v_1$}] {}
	edge (v);
\node [nod2] at (0,-2) (v11) [label=right: {\scriptsize $w=v_1'$}] {}
	edge  (v1);
\node [nod2] at (-1,-2.8) (v112)  {}
	edge  (v11);	
\node [nod2] at (-.4,-2.8) (v112)  {}
	edge  (v11);	
\node [nod2] at (.4,-2.8) (v112)  {}
	edge  (v11);	
\node [nod2] at (1,-2.8) (v112)  {}
	edge  (v11);				
\node [nod2] at (-1,-.5) (v2) [label=above:{\scriptsize $v_2$}] {}
	edge [] (v);		
\node [nod2] at (-2,-1) (v22) [label=above:{\scriptsize $v_2'$}] {}
	edge [] (v2);
\node [nod2] at (-1,.5) (v3)  [label=above:{\scriptsize $v_3$}] {}
	edge [] (v);
\node [nod2] at (-2,1) (v33) [label=above:{\scriptsize $v_3'$}]  {}
	edge [] (v3);
\node [nod2] at (0,1) (v4) [label=right:{\scriptsize $v_4$}]  {}
	edge [] (v);	
\node [nod2] at (0,2) (v44) [label=right:{\scriptsize $v_4'$}]  {}
	edge [] (v4);	
\node [nod2] at (1.5,0) (x) [label=below:{\scriptsize $x$}] {}
	edge (v);
\node [nodr] at (3,0) (u) [label=above left:{\scriptsize $u$}] {}
	edge (x);	
\node [nod2] at (3,-1) (u1) [label=left:{\scriptsize $u_1$}] {}
	edge (u);
\node [nod2] at (3,-2) (u11) [label=left:{\scriptsize $u_1'$}] {}
	edge  (u1);
\node [nod2] at (4,-.5) (u2) [label=above:{\scriptsize $u_2$}] {}
	edge [] (u);		
\node [nod2] at (5,-1) (u22) [label=above:{\scriptsize $u_2'$}] {}
	edge [] (u2);
\node [nod2] at (4,.5) (u3)  [label=above:{\scriptsize $u_3$}] {}
	edge [] (u);
\node [nod2] at (5,1) (u33) [label=above:{\scriptsize $u_3'$}]  {}
	edge [] (u3);
\node [nod2] at (3,1) (u4) [label=left:{\scriptsize $u_4$}]  {}
	edge [] (u);	
\node [nod2] at (3,2) (u44) [label=left:{\scriptsize $u_4'$}]  {}
	edge [] (u4);																		
\end{tikzpicture}  }\hspace*{1cm}
\subfigure[]{
\label{fig:adjacent5-semi}
\begin{tikzpicture}[scale=1]
\node [nodr] at (0,0) (v) [label=above right:{\scriptsize $v$}] {};
\node [nod2] at (0,-1) (v1) [label=right:{\scriptsize $v_1$}] {}
	edge (v);
\node [nod2] at (-.5,-2) (z) [label=left: {\scriptsize $z$}] {}
	edge  (v1);
\node [nod2] at (-.5,-2.8) (y) [label=left: {\scriptsize $y$}] {}
	edge  (z);
\node [nod2] at (0.5,-2) (v11) [label=right: {\scriptsize $w=v_1'$}] {}
	edge  (v1);	
\node [nod2] at (.2,-2.8) (v112)  {}
	edge  (v11);	
\node [nod2] at (.6,-2.8) (v112)  {}
	edge  (v11);	
\node [nod2] at (1,-2.8) (v112)  {}
	edge  (v11);					
\node [nod2] at (-1,-.5) (v2) [label=above:{\scriptsize $v_2$}] {}
	edge [] (v);		
\node [nod2] at (-2,-1) (v22) [label=above:{\scriptsize $v_2'$}] {}
	edge [] (v2);
\node [nod2] at (-1,.5) (v3)  [label=above:{\scriptsize $v_3$}] {}
	edge [] (v);
\node [nod2] at (-2,1) (v33) [label=above:{\scriptsize $v_3'$}]  {}
	edge [] (v3);
\node [nod2] at (0,1) (v4) [label=right:{\scriptsize $v_4$}]  {}
	edge [] (v);	
\node [nod2] at (0,2) (v44) [label=right:{\scriptsize $v_4'$}]  {}
	edge [] (v4);	
\node [nod2] at (1.5,0) (x) [label=below:{\scriptsize $x$}] {}
	edge (v);
\node [nodr] at (3,0) (u) [label=above left:{\scriptsize $u$}] {}
	edge (x);	
\node [nod2] at (3,-1) (u1) [label=left:{\scriptsize $u_1$}] {}
	edge (u);
\node [nod2] at (3,-2) (u11) [label=left:{\scriptsize $u_1'$}] {}
	edge  (u1);
\node [nod2] at (4,-.5) (u2) [label=above:{\scriptsize $u_2$}] {}
	edge [] (u);		
\node [nod2] at (5,-1) (u22) [label=above:{\scriptsize $u_2'$}] {}
	edge [] (u2);
\node [nod2] at (4,.5) (u3)  [label=above:{\scriptsize $u_3$}] {}
	edge [] (u);
\node [nod2] at (5,1) (u33) [label=above:{\scriptsize $u_3'$}]  {}
	edge [] (u3);
\node [nod2] at (3,1) (u4) [label=left:{\scriptsize $u_4$}]  {}
	edge [] (u);	
\node [nod2] at (3,2) (u44) [label=left:{\scriptsize $u_4'$}]  {}
	edge [] (u4);																		
\end{tikzpicture}  }
\caption{$(a)$ Two adjacent bad $5$-vertices. $(b)$ A semi-bad $5$-vertex adjacent to a bad $5$-vertex.}
\label{fig:adjacent}
\end{figure}

%
%

\begin{lem}\label{lem:bad-vertices-not-adjacent}
Let $\D=6$. Then the followings hold.
\begin{itemize} 
\item[$(a)$]  A bad $4$-vertex cannot be adjacent to a bad or a semi-bad vertex.
\item[$(b)$] Two semi-bad $4$-vertices cannot be adjacent.
\item[$(c)$] A bad or a semi-bad $5$-vertex is weak-adjacent to at most one 5-vertex which is a bad or a semi-bad $5$-vertex. 
\end{itemize}
\end{lem}
\begin{proof}
$(a)$ Let $v$ be a bad 4-vertex. Obviously, $v$ cannot be adjacent to a bad or semi-bad $5$-vertex by its definition.   Assume for a contradiction that $v$ is adjacent a bad or a semi-bad $4$-vertex $u$. Let $u_i$, $v_i$ be $2$-neighbour of $u,v$, and let the other neighbour of $u_i$, $v_i$ be $u_i'$, $v_i'$ for $i\in [3]$, respectively. An illustration is depicted in Figure \ref{fig:adjacentbad4} where a bad $4$-vertex $v$ is adjacent to a bad $4$-vertex $u$ in (a) and a semi-bad $4$-vertex $u$ in (b). Notice that $u_1'$ and $u_2'$ are $6$-vertices, and if $u$ is a bad (resp. semi-bad) $4$-vertex, then $u_3'$ is a $6$-vertex (resp. $5$-vertex).

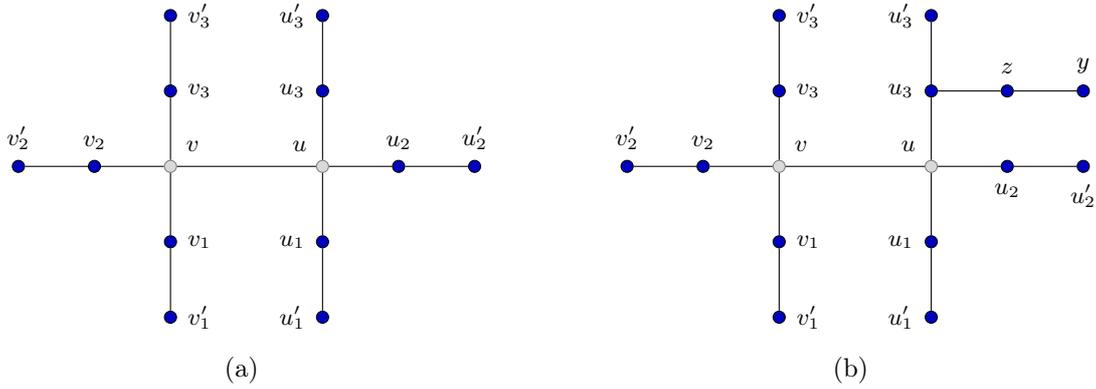
\begin{figure}[htb]
\centering   
\subfigure[]{
\label{fig:bad4ad-a}
\begin{tikzpicture}[scale=1]
\node [nodr] at (0,0) (v) [label=above right:{\scriptsize $v$}] {};
\node [nod2] at (0,-1) (v1) [label=right:{\scriptsize $v_1$}] {}
	edge (v);
\node [nod2] at (0,-2) (v11) [label=right:{\scriptsize $v_1'$}] {}
	edge  (v1);				
\node [nod2] at (-1,0) (v2) [label=above:{\scriptsize $v_2$}] {}
	edge [] (v);		
\node [nod2] at (-2,0) (v22) [label=above:{\scriptsize $v_2'$}] {}
	edge [] (v2);
\node [nod2] at (0,1) (v3) [label=right:{\scriptsize $v_3$}]  {}
	edge [] (v);	
\node [nod2] at (0,2) (v33) [label=right:{\scriptsize $v_3'$}]  {}
	edge [] (v3);	
\node [nodr] at (2,0) (u) [label=above left:{\scriptsize $u$}] {}
	edge (v);	
\node [nod2] at (2,-1) (u1) [label=left:{\scriptsize $u_1$}] {}
	edge (u);
\node [nod2] at (2,-2) (u11) [label=left:{\scriptsize $u_1'$}] {}
	edge  (u1);
\node [nod2] at (3,0) (u2) [label=above:{\scriptsize $u_2$}] {}
	edge [] (u);		
\node [nod2] at (4,0) (u22) [label=above:{\scriptsize $u_2'$}] {}
	edge [] (u2);
\node [nod2] at (2,1) (u3) [label=left:{\scriptsize $u_3$}]  {}
	edge [] (u);	
\node [nod2] at (2,2) (u33) [label=left:{\scriptsize $u_3'$}]  {}
	edge [] (u3);																		
\end{tikzpicture}  }\hspace*{1cm}
\subfigure[]{
\label{fig:bad4ad-b}
\begin{tikzpicture}[scale=1]
\node [nodr] at (0,0) (v) [label=above right:{\scriptsize $v$}] {};
\node [nod2] at (0,-1) (v1) [label=right:{\scriptsize $v_1$}] {}
	edge (v);
\node [nod2] at (0,-2) (v11) [label=right:{\scriptsize $v_1'$}] {}
	edge  (v1);				
\node [nod2] at (-1,0) (v2) [label=above:{\scriptsize $v_2$}] {}
	edge [] (v);		
\node [nod2] at (-2,0) (v22) [label=above:{\scriptsize $v_2'$}] {}
	edge [] (v2);
\node [nod2] at (0,1) (v3) [label=right:{\scriptsize $v_3$}]  {}
	edge [] (v);	
\node [nod2] at (0,2) (v33) [label=right:{\scriptsize $v_3'$}]  {}
	edge [] (v3);	
\node [nodr] at (2,0) (u) [label=above left:{\scriptsize $u$}] {}
	edge (v);	
\node [nod2] at (2,-1) (u1) [label=left:{\scriptsize $u_1$}] {}
	edge (u);
\node [nod2] at (2,-2) (u11) [label=left:{\scriptsize $u_1'$}] {}
	edge  (u1);
\node [nod2] at (3,0) (u2) [label=below:{\scriptsize $u_2$}] {}
	edge [] (u);		
\node [nod2] at (4,0) (u22) [label=below:{\scriptsize $u_2'$}] {}
	edge [] (u2);
\node [nod2] at (2,1) (u3) [label=left:{\scriptsize $u_3$}]  {}
	edge [] (u);	
\node [nod2] at (2,2) (u33) [label=left:{\scriptsize $u_3'$}]  {}
	edge [] (u3);
\node [nod2] at (3,1) (z) [label=above:{\scriptsize $z$}]  {}
	edge [] (u3);	
\node [nod2] at (4,1) (y) [label=above:{\scriptsize $y$}]  {}
	edge [] (z);																				
\end{tikzpicture}  }
\caption{Two adjacent $4$-vertices.}
\label{fig:adjacentbad4}
\end{figure}

By Lemma \ref{lem:4-vertex-color-class}(c), $G-v$ has a 2-distance coloring $f$ with $10$ colors such that the colors $f(v_1'),f(v_2'),$ $f(v_3'),f(u_1),f(u_2),f(u_3)$ are distinct, and $f(N[v_i']-\{v_i\}) = \{1,2,\ldots,6\}$ for $i\in [3]$. So we may assume that $f(v_i')=i$ and $f(u_j)=3+j$ for $i,j\in [3]$. We decolor  $v_1,v_2,v_3,u,u_3$. If $u$ is a semi-bad $4$-vertex (see Figure \ref{fig:bad4ad-b}), we also decolor $z$. 
Since $u_3$ has at most $8$ forbidden colors, it has two available colors, so we assign a color $a\neq 6$ to $u_3$. If $a\in \{1,2,3\}$, then we recolor $u,v_1,v_2,v_3$ in succession, and we give an available color to $v$, since it has at most $9$ forbidden colors.  If $a\in \{7,8,9,10\}$, then we recolor $v_1$ with $a$, and we recolor $u,v_2,v_3,v$ in succession. Besides, if $u$ is a semi-bad $4$-vertex,  we recolor $z$ with an available by Corollary \ref{cor:2-vertex}.  Consequently, $G$ has a 2-distance coloring with $10$ colors, a contradiction.\medskip

$(b)$ Assume to the contrary that $v$ and $u$ are adjacent semi-bad $4$-vertices. Let $u_i$ (resp. $v_i$) be $2$-neighbour of $u$ (resp. $v$), and let the other neighbour of $u_i$, $v_i$ be $u_i'$, $v_i'$ for $i\in [3]$, respectively (see Figure \ref{fig:adjacentsemibad4}).  Notice that $u_1',u_2',v_1',v_2'$ are $6$-vertices, and $u_3',v_3'$ are $5$-vertices.
By Lemma \ref{lem:4-vertex-color-class}(a), $G-v$ has a 2-distance coloring $f$ with $10$ colors such that the colors $f(v_1'),f(v_2'),f(v_3'),f(u_1),f(u_2)$, $f(u_3)$ are distinct, and $f(N[v_i']-\{v_i\}) \subseteq \{1,2,\ldots,6\}$ for $i\in [3]$. So we may assume that $f(v_i')=i$ and $f(u_j)=3+j$ for $i,j\in [3]$. Let us decolor  $v_1,v_2,v_3,u,u_3,z_1,z_2$. 
Since $u_3$ has at most $8$ forbidden colors, it has two available colors, so we assign a color $a\neq 6$ to $u_3$. If $a\in \{1,2,3\}$, then we recolor $u,v_1,v_2,v_3$ in succession, and we give an available color to $v$, since it has at most $9$ forbidden colors.  If $a\in \{7,8,9,10\}$, then we recolor $v_1$ with $a$, and we recolor $u,v_2,v_3,v$ in succession. Finally, we recolor $z_1$ and $z_2$ with some available colors by Corollary \ref{cor:2-vertex}. Consequently, $G$ has a 2-distance coloring with $10$ colors, a contradiction. \medskip

\begin{figure}[htb]
\begin{tikzpicture}[scale=1]
\node [nodr] at (0,0) (v) [label=above right:{\scriptsize $v$}] {};
\node [nod2] at (0,-1) (v1) [label=right:{\scriptsize $v_1$}] {}
	edge (v);
\node [nod2] at (0,-2) (v11) [label=right:{\scriptsize $v_1'$}] {}
	edge  (v1);				
\node [nod2] at (-1,0) (v2) [label=below:{\scriptsize $v_2$}] {}
	edge [] (v);		
\node [nod2] at (-2,0) (v22) [label=below:{\scriptsize $v_2'$}] {}
	edge [] (v2);
\node [nod2] at (0,1) (v3) [label=right:{\scriptsize $v_3$}]  {}
	edge [] (v);	
\node [nod2] at (0,2) (v33) [label=right:{\scriptsize $v_3'$}]  {}
	edge [] (v3);
\node [nod2] at (-1,1) (z1) [label=above:{\scriptsize $z_1$}]  {}
	edge [] (v3);	
\node [nod2] at (-2,1) (y1) [label=above:{\scriptsize $y_1$}]  {}
	edge [] (z1);			
\node [nodr] at (2,0) (u) [label=above left:{\scriptsize $u$}] {}
	edge (v);	
\node [nod2] at (2,-1) (u1) [label=left:{\scriptsize $u_1$}] {}
	edge (u);
\node [nod2] at (2,-2) (u11) [label=left:{\scriptsize $u_1'$}] {}
	edge  (u1);
\node [nod2] at (3,0) (u2) [label=below:{\scriptsize $u_2$}] {}
	edge [] (u);		
\node [nod2] at (4,0) (u22) [label=below:{\scriptsize $u_2'$}] {}
	edge [] (u2);
\node [nod2] at (2,1) (u3) [label=left:{\scriptsize $u_3$}]  {}
	edge [] (u);	
\node [nod2] at (2,2) (u33) [label=left:{\scriptsize $u_3'$}]  {}
	edge [] (u3);
\node [nod2] at (3,1) (z2) [label=above:{\scriptsize $z_2$}]  {}
	edge [] (u3);	
\node [nod2] at (4,1) (y2) [label=above:{\scriptsize $y_2$}]  {}
	edge [] (z2);																				
\end{tikzpicture}  
\caption{Two adjacent semi-bad $4$-vertices.}
\label{fig:adjacentsemibad4}
\end{figure}
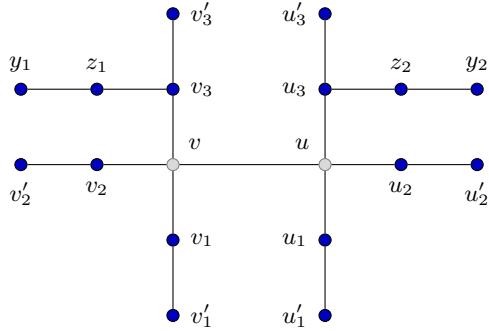


$(c)$ The claim follows from Lemma \ref{lem:bad-5-vertices-not-adjacent}.
\end{proof}

The following is an immediate consequence of Lemma \ref{lem:bad-vertices-not-adjacent}, since two bad or semi-bad vertices cannot be adjacent.

\begin{cor}\label{cor:at-most-lf/2-bad-and-semi-bad}
If $\D=6$, then any face $f$ has at most $\floor[\Big]{ \frac{\ell(f)}{2}}$ bad and semi-bad vertices.
\end{cor}

When $\ell(f)=6$, we can strengthen the statement in Corollary \ref{cor:at-most-lf/2-bad-and-semi-bad} as follows.

\begin{lem}\label{lem:6-face-at-most-3-bad-and-semi-bad}
If $\D=6$, then a face $f$ with $\ell(f)=6$ has at most two bad or semi-bad vertices.
\end{lem}
\begin{proof}
Let $u$ and $v$ be bad or semi-bad vertices belonging to a face $f$ with $\ell(f)=6$. Note that two bad or semi bad vertices cannot be adjacent by Lemma \ref{lem:bad-vertices-not-adjacent}. If $u$ and $v$ have no common neighbour then the claim follows from the fact that $\ell(f)=6$. So we may assume that $u$ and $v$ have a common neighbour $w$.
It is sufficient to show that $f$ has no any other bad or semi bad vertices except for $u$ and $v$.
Assume for a contradiction that $f$ has a bad or a semi-bad vertex $z$ with $z\notin \{u,v\}$. Clearly $z$ is a $4$ or $5$-vertex, in addition, $z$ has a common neighbour with each of $v,u$.   
We will divide the rest of the proof based on the degrees of vertices $u$ and $v$.

Let $u$ and $v$ be $4$-vertices. By Lemma \ref{lem:delta6-special-not-weak-adj}-(1), $u$ and $v$ cannot be weak-adjacent. Additionally,  $u$ and $v$ cannot be star-adjacent as well, since $d(v_3)=5$ in Figure \ref{fig:bad4-b}. Therefore, the common neighbour $w$ has to be $4$ or $5$-vertex, and clearly $w=u_4$ as in Figure \ref{fig:bad4}. If one of $\{u,v\}$ is a bad $4$-vertex, say $v$, then $v$ is not weak-adjacent to a $5^-$-vertex by Lemma \ref{lem:delta6-special-not-weak-adj}-(1), however $z$ is a $4$ or $5$-vertex, a contradiction. Thus we assume that both $u$ and $v$ are semi-bad vertices. Then the vertex $z$ should be $v_3$ as in Figure \ref{fig:bad4-b}, it is of degree $5$. However, a bad or a semi bad $5$-vertex $z$ cannot be weak-adjacent to a $4^-$-vertex by Lemma \ref{lem:delta6-special-not-weak-adj}-(2), a contradiction.

Let $u$ and $v$ be $4$-vertex and $5$-vertex, respectively. This forces that $u$  is a semi-bad $4$-vertex, and $v$ is a semi-bad $5$-vertex such that $u$ and $v$ are star-adjacent by definition.  It follows from Lemma \ref{lem:delta6-special-not-weak-adj} that $z$ is a $5$-vertex. However, $v$ cannot be weak-adjacent to a bad or a semi-bad $5$-vertex $z$ while it is star-adjacent to the $4$-vertex $u$ by Lemma \ref{lem:bad-5-vertices-not-adjacent},  a contradiction.


Finally, let $u$ and $v$ be $5$-vertices. If $w$ is a $3$-vertex, then both $u$ and $v$ are semi-bad $5$-vertices. It follows that all neighbours of $u$ and $v$ except for $w$ are $2$-vertices. We then deduce  that $z$ is a $5$-vertex, and it is weak-adjacent to both $u$ and $v$. However, a bad or a semi-bad $5$-vertex $z$ cannot be weak-adjacent to two bad or semi-bad $5$-vertices $u$ and $v$ by Lemma \ref{lem:bad-vertices-not-adjacent}, a contradiction. We further assume that $w$ is a $2$-vertex. By Lemma \ref{lem:bad-5-vertices-not-adjacent}, $z$ cannot be weak-adjacent to any of $u,v$. This means that
$u$ and $v$ are semi-bad $5$-vertices, and the common neighbour of each pair $\{z,u\}$ and $\{z,v\}$ is a $3$-vertex. However, it is a contradiction with the fact that a bad or semi-bad $5$-vertex $z$ can have at most  one $3$-neighbour.  As a consequence, $f$ has at most two bad or semi bad vertices.
\end{proof}

\begin{lem}\label{lem:7-face-lf-3/2}
Let  $\D = 6$, and let $P=v_1u_1vu_2v_2$ be a path such that $v$ is a bad or a semi-bad vertex and  $u_i$ is a 2-vertex for $i=1,2$. If a face $f$ with $\ell(f)\geq 7$ contains the path $P$, then $f$ has at most $\ceil[\Big]{ \frac{\ell(f)-3}{2}}$ bad or semi-bad vertices. In particular, if $v_1,v_2$ are $6$-vertices, then $f$ has at most $\ceil[\Big]{ \frac{\ell(f)-5}{2}}$ bad or semi-bad vertices.
\end{lem}
\begin{proof}
Let $s_i\in f$ be a neighbour of $v_i$ apart from $u_i$ for $i=1,2$.
If $d(v)=4$, then  $v_1$ and $v_2$ are $6$-vertices by definition, and so we deduce that none of $v_1,v_2,s_1,s_2,u_1,u_2$ can be a bad or semi-bad vertex. Then the claim follows from Lemma \ref{lem:bad-vertices-not-adjacent}.  We further suppose that $v$ is a bad or a semi-bad $5$-vertex. 
Similarly as above, $v_1,v_2$ are $5^+$-vertices by definition. If $v_i$ is  a bad or a semi-bad $5$-vertex, then $v_{3-i}$ cannot be $5$-vertex for $i\in \{1,2\}$ by Lemma \ref{lem:bad-5-vertices-not-adjacent}, so it should be $6^+$-vertex. Notice that a $6^+$-vertex and its neighbours cannot be bad or semi-bad vertices. Thus the claim follows from Lemma \ref{lem:bad-vertices-not-adjacent}. 
Finally we suppose that none of $v_1,v_2$ is a bad or a semi-bad vertex. Then the claim follows from Lemma \ref{lem:bad-vertices-not-adjacent} with the fact that  $\ceil[\Big]{ \frac{\ell(f)-3}{2}}=\ceil[\Big]{ \frac{\ell(f)-5}{2}}+1$.
\end{proof}

We now address some properties of $6$-faces incident to a bad or a semi-bad vertex in order to avoid some configurations.

\begin{lem}\label{lem:partial-face-4-vertex-deg-2}
Let  $\D = 6$, and let $F = \{f_1,f_2\}$ be the partial face set incident to a bad
4-vertex $v$ in Figure \ref{fig:Fbad4-a}, $F_2 = \{f_1\}$ be the partial face set incident to a semi-bad
$4$-vertex $v$ in Figure \ref{fig:Fbad4-b}. If $f_i \in F_1 \ or \ F_2$,  then $s_i\in f_i$ must be
$3^+$-vertex for $i=1,2$, and it is not $3(1)$-vertex. In particular, $s_i$ is not bad or semi-bad vertex.
\end{lem}

\begin{figure}[htb]
\centering   
\subfigure[]{
\label{fig:Fbad4-a}
\begin{tikzpicture}[scale=1]
\node [nodr] at (0,0) (v) [label=above:{\scriptsize $v$}] {};
\node [nod2] at (-1,-1) (u1) [label=above:{\scriptsize $u_1$}] {}
	edge (v);
\node [nod2] at (-2,-1) (v1) [label=above:{\scriptsize $v_1$}] {}
	edge  (u1);
\node [nod2] at (-1,1) (u2) [label=below:{\scriptsize $u_2$}] {}
	edge [] (v);		
\node [nod2] at (-2,1) (v2) [label=above:{\scriptsize $v_2$}] {}
	edge [] (u2);
\node [nod2] at (1,1) (u3)  [label=below:{\scriptsize $u_3$}] {}
	edge [] (v);
\node [nod2] at (2,1) (v3) [label=above:{\scriptsize $v_3$}]  {}
	edge [] (u3);
\node [nod2] at (1,-1) (u4) [label=above:{\scriptsize $u_4$}] {}
	edge [] (v);
\node [nod2] at (2,0) (v7) [label=right:{\scriptsize $v_7$}] {}
	edge [dotted] (u4);	
\node [nod2] at (2,-.65) (v6)  [label=right:{\scriptsize $v_6$}]  {}
	edge [] (u4);
\node [nod2] at (2,-1.3) (v5)  [label=right:{\scriptsize $v_5$}]  {}
	edge [] (u4);
\node [nod2] at (2,-2) (v4)  [label=right:{\scriptsize $v_4$}] {}
	edge [] (u4);
\node [nod2] at (-3,0) (s1)  [label=left:{\scriptsize $s_1$}] {}
	edge [] (v1)
	edge [] (v2);	
\node [nod2] at (0,2) (s2)  [label=above:{\scriptsize $s_2$}] {}
	edge [] (v2)
	edge [] (v3);
\node at (-1.5,0) (v11)  {\scriptsize $f_1$};	
\node at (0,1.2) (v22)  {\scriptsize $f_2$};								
\end{tikzpicture}  }\hspace*{1cm}
\subfigure[]{
\label{fig:Fbad4-b}
\begin{tikzpicture}[scale=1]
\node [nodr] at (0,0) (v) [label=above:{\scriptsize $v$}] {};
\node [nod2] at (-1,-1) (u1) [label=above:{\scriptsize $u_1$}] {}
	edge (v);
\node [nod2] at (-2,-1) (v1) [label=above:{\scriptsize $v_1$}] {}
	edge  (u1);
\node [nod2] at (-1,1) (u2) [label=below:{\scriptsize $u_2$}] {}
	edge [] (v);		
\node [nod2] at (-2,1) (v2) [label=above:{\scriptsize $v_2$}] {}
	edge [] (u2);
\node [nod2] at (1,1) (u3)  [label=below:{\scriptsize $u_3$}] {}
	edge [] (v);
\node [nod2] at (2,1.4) (v3) [label=above:{\scriptsize $v_3$}]  {}
	edge [] (u3);
\node [nod2] at (2,.6) (z) [label=above:{\scriptsize $z$}]  {}
	edge [] (u3);	
\node [nod2] at (3,.6) (y) [label=above:{\scriptsize $y$}]  {}
	edge [] (z);		
\node [nod2] at (1,-1) (u4) [label=above:{\scriptsize $u_4$}] {}
	edge [] (v);
\node [nod2] at (2,0) (v7) [label=right:{\scriptsize $v_7$}] {}
	edge [dotted] (u4);	
\node [nod2] at (2,-.65) (v6)  [label=right:{\scriptsize $v_6$}]  {}
	edge [] (u4);
\node [nod2] at (2,-1.3) (v5)  [label=right:{\scriptsize $v_5$}]  {}
	edge [] (u4);
\node [nod2] at (2,-2) (v4)  [label=right:{\scriptsize $v_4$}] {}
	edge [] (u4);	
\node [nod2] at (-3,0) (s1)  [label=left:{\scriptsize $s_1$}] {}
	edge [] (v1)
	edge [] (v2);
\node [nod2] at (0,2) (s2)  [label=above:{\scriptsize $s_2$}] {}
	edge [] (v2)
	edge [] (v3);	
\node at (-1.5,0) (v11)  {\scriptsize $f_1$};
\end{tikzpicture}  }
\caption{Partial faces incident to a bad $4$-vertex $(a)$ and a semi-bad $4$-vertex $(b)$.}
\label{fig:Fbad4}
\end{figure}
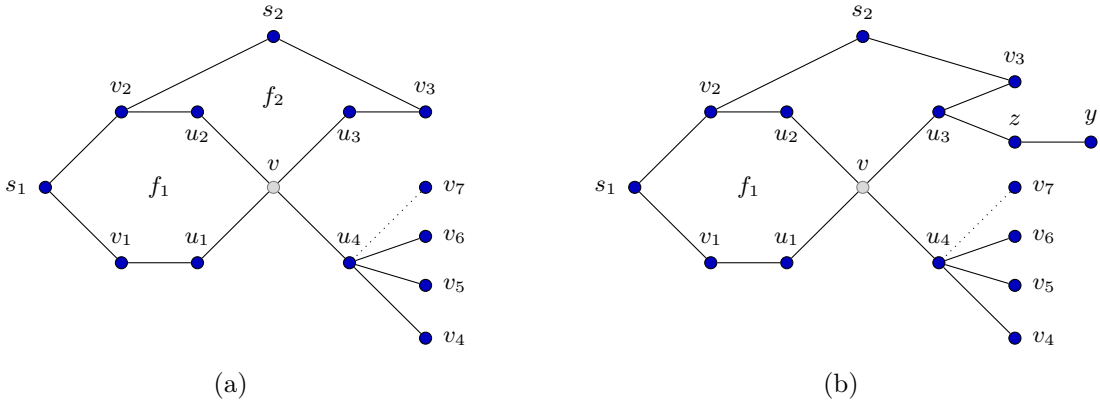

\begin{proof}
Assume to the contrary that $s\in \{s_1,s_2\}$ is a $2$-vertex or a $3(1)$-vertex. By minimality,  $G-v$ has a $2$-distance coloring $f$ with $10$ colors, and we may assume that $f(N[v_i]-\{u_i\}) \subseteq \{1,2,\ldots,6\}$ for  $i\in [3]$ by Lemma~\ref{lem:4-vertex-color-class}(a)-(d), so $f(s) \in \{ 1,2,\ldots,6\}$. Assume without loss of generality that $f_i = f_1= v_1u_1vu_2v_2s_1$. Recall that $d(u_1)=d(u_2)=2$, and the vertices $v_1,v_2,v_3$ in Figure~\ref{fig:Fbad4-a} and $v_1,v_2$ in Figure~\ref{fig:Fbad4-b} are $6$-vertices  by definition. We decolor $s_1,u_1,u_2,u_3$. In particular, if $s_1$ is a $3(1)$-vertex, then we also decolor the $2$-neighbour of $s_1$. 
Notice that $v$ has currently at most $9$ forbidden colors, so we give an available color to it.
We later recolor $u_3,u_2,u_1$ in succession, respectively. Finally, let us consider the vertex $s_1$, if $d(s_1)=2$, then  $s_1$ has  at most $7$ forbidden colors due to $f(N[v_i]-\{u_i\}) \subseteq \{1,2,\ldots,6\}-f(s_1)$ for $i\in \{1,2\}$, so we recolor it. Otherwise, if $s_1$ is a $3(1)$-vertex, then $s_1$ has  at most $8$ forbidden colors, so we recolor $s_1$ and the $2$-neighbour of $s_1$ in succession.  This provides a 2-distance coloring of $G$ with $10$ colors, a contradiction. On the other hand, $s_1$ cannot be bad or semi-bad vertex, since it has two $6$-neighbours.
\end{proof}

\begin{lem}\label{lem:partial-face-4-vertex-deg-3(1)}
Let  $\D = 6$, and let $f'=v_2u_2vu_3v_3s_2$ be a face incident to a semi-bad
4-vertex $v$ in Figure \ref{fig:Fbad4-b}. Then $s_2$ must be $3^+$-vertex, and it is not $3(1)$-vertex. In particular, $s_2$ is not bad or semi-bad vertex.
\end{lem}
\begin{proof}
Assume that $s_2$ is a 2-vertex or a $3(1)$-vertex. By minimality,  $G-v$ has a 2-distance coloring $f$ with $10$ colors, 
and we may assume that $f(N[v_i]-\{u_i\}) \subseteq \{1,2,\ldots,6\}$ for  $i\in [3]$ by Lemma \ref{lem:4-vertex-color-class}-(a)-(b), so $f(s_2) \in \{ 1,2,\ldots,6\}$. Recall that  $v_2$ is a $6$-vertex, and $u_3$ has a $2$-neighbour $z$ and a $5$-neighbour $v_3$.   We decolor $s_2,u_1,u_2,u_3,z$. In particular, if $s_2$ is a $3(1)$-vertex, then we also decolor the $2$-neighbour of $s_2$. We later recolor  $v,u_1,u_2,u_3,z$ in succession, respectively. If $d(s_2)=2$, then  $s_2$ has  at most $8$ forbidden colors, so we recolor it. Otherwise, if $s_2$ is a $3(1)$-vertex, then $s_2$ has  at most $9$ forbidden colors, so we recolor $s_2$ and the  $2$-neighbour of $s_2$ in succession.  This provides a $2$-distance coloring of $G$ with $10$ colors, a contradiction. Additionally, $s_2$ cannot be bad or semi-bad vertex, since it has a $6$-neighbour.
\end{proof}

We now turn our attention to the $6$-faces incident to a bad or a semi-bad $5$-vertex.

\begin{lem}\label{lem:partial-face-5-vertex-deg-5}
Let  $\D = 6$, and let $F_1 = \{f_1,f_2,f_3,f_4,f_5\}$ be the face set incident to a bad
$5$-vertex in Figure \ref{fig:Fbad5-a}, $F_2 = \{f_1,f_2,f_3\}$ be the partial face set incident to a semi-bad $5$-vertex in Figure \ref{fig:Fbad5-b}. If $f_i \in F_1 \ or \ F_2$, $f_i = v_iu_ivu_{i+1}v_{i+1}s_i$ and $\min\{d(v_i),d(v_{i+1})\}=5$, then $s_i$ must be a $3^+$-vertex, and it is not $3(1)$-vertex. In particular, $s_i$ is not bad or semi-bad vertex.
\end{lem}

\begin{figure}[htb]
\centering   
\subfigure[]{
\label{fig:Fbad5-a}
\begin{tikzpicture}[scale=1]
\node [nodr] at (0,0) (v) [label=above:{\scriptsize $v$}] {};
\node [nod2] at (-1,0) (u1) [label=above:{\scriptsize $u_1$}] {}
	edge (v);
\node [nod2] at (-2,0) (v1) [label=left:{\scriptsize $v_1$}] {}
	edge  (u1);
\node [nod2] at (-1,1) (u2) [label=above:{\scriptsize $u_2$}] {}
	edge [] (v);		
\node [nod2] at (-1.7,1.7) (v2) [label=above:{\scriptsize $v_2$}] {}
	edge [] (u2);
\node [nod2] at (1,1) (u3)  [label=above:{\scriptsize $u_3$}] {}
	edge [] (v);
\node [nod2] at (1.7,1.7) (v3) [label=above:{\scriptsize $v_3$}]  {}
	edge [] (u3);
\node [nod2] at (1,0) (u4) [label=above:{\scriptsize $u_4$}] {}
	edge [] (v);
\node [nod2] at (2,0) (v4) [label=right:{\scriptsize $v_4$}] {}
	edge [] (u4);	
\node [nod2] at (0,-.8) (u5) [label=right:{\scriptsize $u_5$}] {}
	edge [] (v);
\node [nod2] at (0,-1.6) (v5) [label=below:{\scriptsize $v_5$}] {}
	edge [] (u5);		
\node at (1,-2.55) (asd)     {};

\node [nod2] at (-2.5,1) (s1)  [label=left:{\scriptsize $s_1$}] {}
	edge [] (v1)
	edge [] (v2);	
\node [nod2] at (0,2.5) (s2)  [label=above:{\scriptsize $s_2$}] {}
	edge [] (v2)
	edge [] (v3);
\node [nod2] at (2.5,1) (s2)  [label=right:{\scriptsize $s_3$}] {}
	edge [] (v3)
	edge [] (v4);
\node [nod2] at (1.5,-1.6) (s2)  [label=below:{\scriptsize $s_4$}] {}
	edge [] (v4)
	edge [] (v5);
\node [nod2] at (-1.5,-1.6) (s2)  [label=below:{\scriptsize $s_5$}] {}
	edge [] (v5)
	edge [] (v1);			
\node at (-1.8,.8) (v11)  {{\scriptsize $f_1$}};	
\node at (0,1.5) (v22)  {{\scriptsize $f_2$}};
\node at (1.8,.8) (v22)  {{\scriptsize $f_3$}};	
\node at (1,-.8) (v22)  {{\scriptsize $f_4$}};
\node at (-1,-.8) (v22)  {{\scriptsize $f_5$}};
				
\end{tikzpicture}  }\hspace*{1cm}
\subfigure[]{
\label{fig:Fbad5-b}
\begin{tikzpicture}[scale=1]
\node [nodr] at (0,0) (v) [label=above:{\scriptsize $v$}] {};
\node [nod2] at (-1,0) (u1) [label=below:{\scriptsize $u_1$}] {}
	edge (v);
\node [nod2] at (-2,0) (v1) [label=below:{\scriptsize $v_1$}] {}
	edge  (u1);
\node [nod2] at (-1,1) (u2) [label=above:{\scriptsize $u_2$}] {}
	edge [] (v);		
\node [nod2] at (-1.7,1.7) (v2) [label=above:{\scriptsize $v_2$}] {}
	edge [] (u2);
\node [nod2] at (1,1) (u3)  [label=above:{\scriptsize $u_3$}] {}
	edge [] (v);
\node [nod2] at (1.7,1.7) (v3) [label=above:{\scriptsize $v_3$}]  {}
	edge [] (u3);
\node [nod2] at (1,0) (u4) [label=below:{\scriptsize $u_4$}] {}
	edge [] (v);
\node [nod2] at (2,0) (v4) [label=below:{\scriptsize $v_4$}] {}
	edge [] (u4);	
\node [nod2] at (0,-.8) (u5) [label=right:{\scriptsize $u_5$}] {}
	edge [] (v);
\node [nod2] at (-1,-1.6) (z) [label=left:{\scriptsize $z$}] {}
	edge [] (u5);
\node [nod2] at (-1,-2.4) (y) [label=left:{\scriptsize $y$}] {}
	edge [] (z);
\node [nod2] at (1,-1.6) (v5) [label=left:{\scriptsize $v_5$}] {}
	edge [] (u5);	
\node [nod2] at (2,-.8) (v51) {}
	edge [dotted] (v5);	
\node [nod2] at (2,-1.4) (v52) {}
	edge [] (v5);
\node [nod2] at (2,-2) (v53) {}
	edge [] (v5);
\node [nod2] at (2,-2.6) (v54) {}
	edge [] (v5);											
\node [nod2] at (-2.5,1) (s1)  [label=left:{\scriptsize $s_1$}] {}
	edge [] (v1)
	edge [] (v2);	
\node [nod2] at (0,2.5) (s2)  [label=above:{\scriptsize $s_2$}] {}
	edge [] (v2)
	edge [] (v3);
\node [nod2] at (2.5,1) (s2)  [label=right:{\scriptsize $s_3$}] {}
	edge [] (v3)
	edge [] (v4);		
\node at (-1.8,.8) (v11)  {{\scriptsize $f_1$}};	
\node at (0,1.5) (v22)  {{\scriptsize $f_2$}};
\node at (1.8,.8) (v22)  {{\scriptsize $f_3$}};	
\end{tikzpicture}  }
\caption{Partial faces incident to a bad $5$-vertex $(a)$ and a semi-bad $5$-vertex $(b)$.}
\label{fig:Fbad5}
\end{figure}
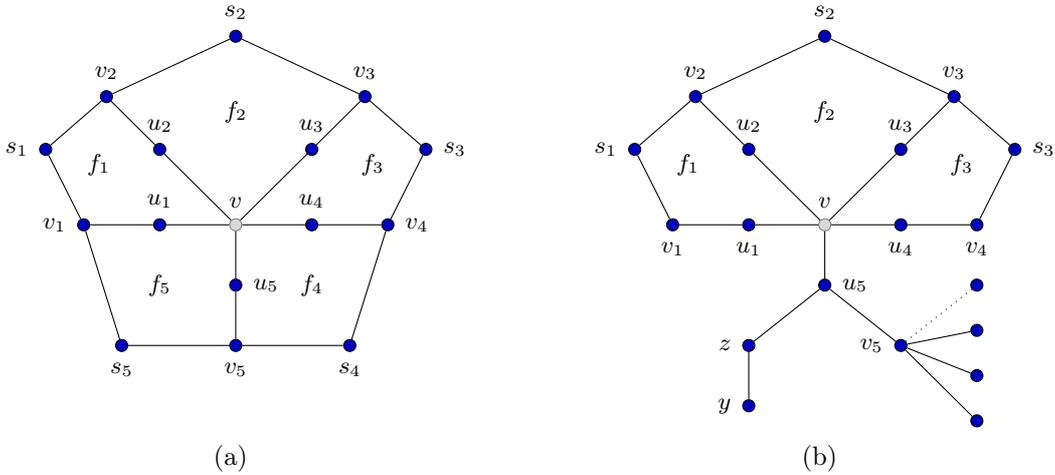

\begin{proof}
Assume without loss of generality that $f_i = f_1= v_1u_1vu_2v_2s_1$. By contradiction, assume that $s_1$ is a $2$-vertex or a $3(1)$-vertex. Note that  $G-v$ has a 2-distance coloring $f$ with $10$ colors by minimality, and so we may assume that $\{1,2,\ldots,5\} \subseteq f(N[v_i]-\{u_i\}) $ for  $i\in [2]$ by Lemma \ref{lem:5-vertex-color-class}.    Without loss of generality, we may assume that $\min\{d(v_1),d(v_2)\}=d(v_1)=5$. Note that $f(s_1)\in \{ 1,2,\ldots,5\}$, since $d(v_1)=5$ and  $\{1,2,\ldots,5\} \subseteq f(N[v_1]-\{u_1\}) $. We decolor $s_1,u_1,u_2,\ldots,u_5$. In particular, if $s_1$ (resp. $u_5$) is a $3(1)$-vertex, then we also decolor the $2$-neighbour of $s_1$ (resp. $u_5$). Now, we recolor   $u_2,u_3,u_4,u_5,v$ in succession, respectively. 
Next, we recolor $u_1$ with an available color, since it has at most $9$ forbidden colors. Finally, consider the vertex $s_1$, if $d(s_1)=2$, then  $s_1$ has  at most $8$ forbidden colors, so we recolor it. Otherwise, if $s_1$ is a $3(1)$-vertex, then $s_1$ has  at most $9$ forbidden colors, so we recolor $s_1$ and the 2-neighbour of $s_1$ in succession. Finally, if $u_5$ is a $3(1)$-vertex, then we recolor its $2$-neighbour with an available color by Corollary \ref{cor:2-vertex}.  This provides a 2-distance coloring of $G$ with $10$ colors, a contradiction. In particular, the vertex $s_1$ has two $5^+$-neighbour, so it cannot be bad or semi-bad vertex by the definition.
\end{proof}

If a vertex is $2$-vertex or $3(1)$-vertex, then we call it as a  \emph{poor vertex}.  A $5$-vertex is said to be \emph{senior} if it has at most $4$ poor neighbours.\medskip

\begin{lem}\label{lem:partial-face-5-vertex-no-poor}
Let  $\D = 6$, and let $\{f_i,f_{i+1}\}$ be the partial face set incident to a bad or semi-bad
5-vertex $v$ in Figure \ref{fig:Fbad5}. Then one of  $s_i,s_{i+1}$ is not poor vertex. 
\end{lem}
\begin{proof}
Without loss of generality, let $i=1$. Suppose by contradiction that both $s_1$ and $s_2$ are poor vertices. By minimality,  $G-v$ has a 2-distance coloring $f$ with $10$ colors, and we may assume that $ \{1,2,\ldots,5\}\subseteq f(N[v_i]-\{u_i\}) $ for  $i\in [3]$ by Lemma \ref{lem:5-vertex-color-class}.  Notice first that if $v_1$ or $v_2$ is a $5$-vertex, then $s_1$ cannot be a poor vertex by Lemma \ref{lem:partial-face-5-vertex-deg-5}. We therefore assume that  $\min\{d(v_1),d(v_2)\}=6$. We decolor $s_1,s_2,u_1,u_2,\ldots,u_5$.  In particular, if $s_1$ (resp. $s_2$, $u_5$) is a $3(1)$-vertex, then we also decolor the $2$-neighbour of $s_1$ (resp. $s_2$, $u_5$).
We now recolor $u_5,u_4,u_3,u_1,v$ in succession, respectively. At last, we recolor $u_2$ with an available color, since it has at most $9$ forbidden colors. Finally, consider the vertex $s_i$ for each $i=2,3$, if $d(s_i)=2$, then  $s_i$ has  at most $8$ forbidden colors, so we recolor it. Otherwise, if $s_i$ is a $3(1)$-vertex, then $s_i$ has  at most $9$ forbidden colors, so we recolor $s_i$ and the 2-neighbour of $s_i$ in succession. On the other hand, if $u_5$ is a $3(1)$-vertex, we recolor its $2$-neighbour with an available color by Corollary \ref{cor:2-vertex}.  This provides a 2-distance coloring of $G$ with $10$ colors, a contradiction.
\end{proof}

As a consequence of Lemma \ref{lem:partial-face-5-vertex-no-poor} together with Lemma \ref{lem:partial-face-5-vertex-deg-5}, we conclude the following.

\begin{cor}\label{cor:delta6-no-consecutive}
Let  $\D = 6$, and let $F = \{f_1,f_2,f_3,f_4,f_5\}$ be the partial face set incident to a bad or a semi-bad $5$-vertex $v$ in Figure \ref{fig:Fbad5}.  Then the following holds.
\begin{itemize}
\item[$(1)$] No consecutive two vertices of $s_i,s_{i+1}$ can be poor vertices.
\item[$(2)$] If $s_i$ is a poor vertex, then $\min\{d(v_i),d(v_{i+1})\}=6$.
\end{itemize}
\end{cor}

\begin{lem}\label{lem:delta6-f-is-not-inc-to-bad-and-semibad}
Let  $\D = 6$, and let $f_i=v_iu_ivu_{i+1}v_{i+1}s$ be a $6$-face incident to a bad or a semi-bad vertex $v$ such that $u_j$ is a $2$-vertex for $j\in \{1,2\}$ (see Figure \ref{fig:Fbad5}). Then $f_i$ is not incident to a bad or a semi-bad vertex except for $v$.  
\end{lem}
\begin{proof}
We may assume without loss of generality that $i=1$, and so $f_1=v_1u_1vu_2v_2s$.
By contradiction, assume that $f_1$ is incident to a bad or a semi-bad vertex except for $v$. If $d(v)=4$, then  $v_1$ and $v_2$ are $6$-vertices by definition. We then conclude that none of $v_1,v_2,s$ can be a bad or a semi bad vertex.  We may further suppose that $v$ is a bad or a semi-bad $5$-vertex. Again, $v_1$ and $v_2$ are $5^+$-vertices by definition. If $s$ is a poor vertex, then both $v_1$ and $v_2$ are $6$-vertices by Lemma \ref{lem:partial-face-5-vertex-deg-5}. This implies that none of $v_1,v_2$ can be bad or semi-bad vertex, and so $f_1$ has only one bad or a semi-bad vertex.  We further suppose that $s$ is not poor vertex, i.e., it is either a $3(0)$-vertex or a $4^+$-vertex. By the definition of a bad and a semi-bad vertices, we deduce that none of  $v_1,v_2$ can be a bad or a semi-bad vertex in both cases. Thus, the claim holds.
\end{proof}

\noindent
We call a  path $wuvzt$ as a \emph{crucial path} if it satisfies one of the following.
\begin{itemize}
\item $d(w)=d(t)=\D$, \ $d(u)=d(z)=2$, and $d(v)=4$,
\item $d(w)=d(t)=\D$, \ $d(u)=d(z)=2$, and $d(v)=5$,
\item $d(w)=\D$, $d(u)=2$, $d(v)=4$, $d(z)=5$, and  $d(t)\geq 6$. \medskip
\end{itemize} 

If a crucial path $wuvzt$ lies on the boundary of a face $f$, we call the vertex $v$ as a \emph{$f$-crucial vertex}. It follows from the definition of the crucial path that we can bound the number of those vertices in a face.

\begin{cor}\label{cor:at-most-lf-4}
Any face $f$ with $\ell(f)\geq 6$ has at most $\floor[\Big]{ \frac{\ell(f)}{4}}$ $f$-crucial vertices.
\end{cor}

In the rest of the paper, we will apply discharging to show that $G$ does not exist. We assign to each vertex $v$ a charge $\mu(v)=2d(v)-6$ and to each face $f$ a charge $\mu(f)=\ell(f)-6$. By Euler's formula, we have

$$\sum_{v\in V}(2d(v)-6)+\sum_{f\in F}(\ell(f)-6)=-12$$

We next present some rules and redistribute accordingly. Once the discharging finishes, we check the final charge $\mu^*(v)$ and $\mu^*(f)$. If $\mu^*(v)\geq 0$ and $\mu^*(f)\geq 0$, we get a contradiction that no such a counterexample can exist.


\subsection{Discharging Rules} \label{sub:} ~~\medskip

\noindent
\textbf{Case 1:} Let $\Delta\geq 7$. \medskip

We apply the following discharging rules.

\begin{itemize}
\item[\textbf{R1:}] Every $2$-vertex receives $1$  from each neighbour.
\item[\textbf{R2:}] Every $3(1)$-vertex receives $1$  from each $6^+$-neighbour.
\item[\textbf{R3:}] Every $3(1)$-vertex receives $\frac{1}{2}$  from each $5$-neighbour.
\item[\textbf{R4:}] Every $4(3)$-vertex receives $\frac{1}{2}$  from each $5(k)$-neighbour for $0 \leq k \leq 3$.
\item[\textbf{R5:}] Every $6^+$-vertex gives $1$  to each $4^-$-neighbour.
\item[\textbf{R6:}] Every $7^+$-vertex transfers its positive charge (after R5) equally to its incident faces. 
\item[\textbf{R7:}] Every face $f$ transfers its positive charge equally to  its incident $f$-crucial vertices.\medskip
\end{itemize}

\begin{rem}\label{rem:1/7-charge}
By R5, if $v$ is a $7^+$-vertex, then $v$ gives 1  to its each $4^-$-neighbour. It then follows from R6, $v$ gives $\frac{d(v)-6}{d(v)}\geq \frac{1}{7}$  to each incident faces when all neighbour of $v$ are $4^-$-vertices. If one of neighbour of $v$ is a $5^+$-vertex, then $v$ gives at least  $\frac{d(v)-5}{d(v)}\geq \frac{2}{7}$  to each incident face. 
\end{rem}


\noindent
\textbf{Checking} $\mu^*(v), \mu^*(f)\geq 0$, for $v\in V(G), f\in F(G)$\medskip

Clearly $\mu^*(f)\geq 0$ for each $f\in F(G)$, since every face transfers its positive charge equally to its incident $f$-crucial vertices. \medskip

We pick a vertex $v\in V(G)$ with $d(v)=k$. \medskip

\textbf{(1).} Let $k=2$.  Recall that all neighbours of a $2$-vertex are $3^+$-vertices by Proposition \ref{prop:properties}. Then $\mu^*(v)=0$ by R1. \medskip

\textbf{(2).} Let $k=3$. By Corollary \ref{cor:3-vertex}-(1), $v$ has at most one 2-neighbour. If $v$ has no 2-neighbour, then $\mu^*(v)\geq 0$, since it does not give a charge to any vertex. Otherwise, $v$ has either a $6^+$-neighbour or two $5$-neighbours by Corollary \ref{cor:3-vertex}-(3), and so $v$ gets at least $1$  from its neighbours by applying R3 and R5. Therefore, $\mu^*(v)\geq 0$  after $v$ transfers $1$  to its $2$-neighbour by R1.  \medskip  

\textbf{(3).} Let $k=4$. Observe that $\mu^*(v)\geq 0$ if $v$ has no three $2$-neighbours or has a $6^+$-neighbour by R1 and R5 together with Corollary \ref{cor:4-vertex}-(1). It then follows from Corollary \ref{cor:4-vertex}-(1) that we may assume that $v$ has three $2$-neighbours, and all neighbours of $v$ are $5^-$-vertices.  Denote by $u_1,u_2,u_3$  the $2$-neighbours of $v$. Note that the other neighbour of $v$ must be $5$-vertex by Corollary \ref{cor:4-vertex}-(5), say $u_4$. 
Clearly, $v$ is a special vertex of type II. Moreover, $v$ is not weak-adjacent to any $(\D-1)^-$-vertex by Lemma \ref{lem:delta7-not-weak-delta-1}. Therefore, the other neighbour of each $u_i$ for $i\in [3]$ is a $\D$-vertex, except for $v$. Denote by $v_i$ the other neighbour of $u_i$ for $i\in [3]$.
We may assume, without loss of generality, that the neighbours of $v$ ($u_1, u_2, u_3, u_4$) are arranged in clockwise order on the plane. Let $f_1$ be the face incident to $u_1,v,u_2$, and let $f_2$ be the face incident to $u_2,v,u_3$. Notice that $v$ is both $f_1$-crucial vertex and  $f_2$-crucial vertex. Since $v_i$ is a $\D$-vertex with $\D\geq 7$ for $i\in [3]$, it gives $\frac{1}{7}$  to each incident face by R6 and Remark \ref{rem:1/7-charge}. Observe that each $f_j$, for $j\in \{1,2\}$, gets totally at least $2\times \frac{1}{7}$ from $v_j$ and $v_{j+1}$. 
After applying R7, the vertex $v$ gets totally at least $\frac{4}{7}$  from $f_1$ and $f_2$ by Corollary \ref{cor:at-most-lf-4}. 

If $u_4$ has at most three 2-neighbour, then $v$ receives  $\frac{1}{2}$  from $u_4$ by R4. Thus we have $\mu^*(v)\geq 2d(v)-6+\frac{4}{7}+\frac{1}{2}-3>0$ after $v$ transfers $1$  to its each $2$-neighbour by R1.
We further  assume that $u_4$ is a $5(4)$-vertex. In this case, by Proposition \ref{prop:type2-vertex5}, $u_4$ is not weak-adjacent to any $5^-$-vertex. If the face incident to $v_3,u_3,v,u_4$, say $f_3$, is a $6$-face, then $v_3$ gives $\frac{2}{7}$  to its each incident face by R6 together with Proposition \ref{prop:properties} and Remark \ref{rem:1/7-charge}. So, each of $f_2,f_3$ receives at least $\frac{2}{7}$ from $v_3$, and transfers it to $v$ by R7.  Recall that $v$ has already received at least $\frac{4}{7}$  from $f_1$ and $f_2$ as stated above. Consequently, $v$ gets totally at least $1$  from $f_1,f_2,f_3$ by R7 and Corollary \ref{cor:at-most-lf-4}. Otherwise, if the face $f_3$ incident to $v_3,u_3,v,u_4$ is a $7^+$-face, then $\mu(f_3)=\ell(f_3)-6\geq 1$ and $f_3$ has  at most $\floor[\Big]{ \frac{\ell(f)}{4}}$ $f_3$-crucial vertex by Corollary \ref{cor:at-most-lf-4}. So, $v$ gets at least $1$  from $f_3$ by R7. Consequently, $\mu^*(v)> 0$. \medskip

\textbf{(4).} Let $k=5$. Observe that $\mu^*(v)\geq 0$ if $v$ has at most three $2$-neighbours by R1, R3 and R4. If $v$ has four $2$-neighbours and one $3^+$-neighbour different from $3(1)$-vertex, then $\mu^*(v)\geq 0$ by applying R1. So we have two cases; $v$ is either a $5(4)$-vertex having a $3(1)$-neighbour or a $5(5)$-vertex. Let $u_1,u_2,u_3,u_4$ be $2$-neighbours of $v$ in a clockwise ordering on the plane. Let $u_5$ be the last neighbour of $v$, which is either a $2$-vertex or a $3(1)$-vertex. By Lemma \ref{lem:delta7-not-weak-delta-1}, 
$v$ is not weak-adjacent to any $(\D-1)^-$-vertex. Then the other neighbour of each $u_i$ is $\Delta$-vertex.  Denote by $v_i$ the other neighbour of $u_i$ for $i\in [4]$. Since $v_i$ is a $\D$-vertex with $\D\geq 7$, the vertex $v_i$ gives at least $\frac{1}{7}$  to its each incident face by R6 and Remark \ref{rem:1/7-charge}. Denote by $f_i$ the face containing the path $v_iu_ivu_{i+1}v_{i+1}$ for $i\in [3]$. Notice that each of $f_1,f_2,f_3$ gets totally at least $\frac{2}{7}$ from $v_i$'s by R6. It then follows that each of $f_1,f_2,f_3$ gives $\frac{2}{7}$  to $v$ by R7 and Corollary \ref{cor:at-most-lf-4}. So $v$ gets totally at last $\frac{6}{7}$  from $f_1,f_2,f_3$. If $u_5$ is a $3(1)$-vertex, then we have $\mu^*(v)\geq 0$ after $v$ transfers $1$  to its each $2$-neighbour by R1 and $\frac{1}{2}$  to $u_5$ by R3. Otherwise, if  $u_5$ is not a $3(1)$-vertex, then it must be a $2$-vertex, say $N(u_5)=\{v,v_5\}$. Consider the face containing $v_4u_4vu_5v_5$, say $f_4$,  similarly as above, $f_4$ gets totally at least $\frac{2}{7}$  from $v_4,v_5$ by R6. It follows that $f_4$ gives $\frac{2}{7}$  to $v$ by R7 and Corollary \ref{cor:at-most-lf-4}. Consequently, $v$ gets totally at least $\frac{8}{7}$  from $f_1,f_2,f_3,f_4$.  Thus, $\mu^*(v)> 0$  after $v$ transfers $1$  to its each $2$-neighbour by R1. \medskip

\textbf{(5).} Let $k\geq 6 $. By R5, $v$ gives at most $1$  to its each neighbour, and so $\mu^*(v)\geq 0$.\medskip

\noindent
\textbf{Case 2:} Let $\Delta=6$. \medskip

We apply the following discharging rules.

\begin{itemize}
\item[\textbf{R1:}] Every $2$-vertex receives $1$  from each neighbour.
\item[\textbf{R2:}] Every $3(1)$-vertex receives $1$  from each  $6^+$-neighbour and senior $5$-neighbour. 
\item[\textbf{R3:}] If a $3(1)$-vertex has no $6^+$-neighbour or senior $5$-neighbour, then $v$ receives $\frac{1}{2}$  from its each $4^+$-neighbour.
\item[\textbf{R4:}] If a $5$-vertex $v$ is adjacent to at most three poor vertices, then $v$ transfers  its positive charge (after R1,R2,R3) equally to its incident faces. 
\item[\textbf{R5:}] Every $6^+$-vertex gives $1$  to each $2$-neighbour, $3(1)$-neighbour, $4(3)$-neighbour, and $4(2)$-neighbour which is adjacent to a $3(1)$-vertex.
\item[\textbf{R6:}] Every $6^+$-vertex transfers its positive charge (after R5) equally to its incident faces. 
\item[\textbf{R7:}] Every face transfers its positive charge equally to  its incident bad and semi-bad vertices.\medskip
\end{itemize}

\noindent
\textbf{Checking} $\mu^*(v), \mu^*(f)\geq 0$, for $v\in V(G), f\in F(G)$\medskip

Clearly $\mu^*(f)\geq 0$ for each $f\in F(G)$ as every face transfers its positive charge equally to its incident bad and semi-bad vertices by R7. \medskip

We pick a vertex $v\in V(G)$ with $d(v)=k$. \medskip

\textbf{(1).} Let $k=2$. Recall that all neighbours of a $2$-vertex are $3^+$-vertices by Proposition \ref{prop:properties}. Then $\mu^*(v)=0$ by R1. \medskip

\textbf{(2).} Let $k=3$. By Corollary \ref{cor:3-vertex}-(1), $v$ has at most one 2-neighbour. If $v$ has no such neighbour, then $\mu^*(v)\geq 0$, since it does not give a charge to any vertex. Otherwise, suppose that $v$ has a 2-neighbour, we will show that $v$ always gets totally 1  from its neighbours. If $v$ has a $6^+$-neighbour $u$, then  $u$ gives 1  to $v$ by R5. If $v$ has no $6^+$-neighbour, then $v$ has at least one $5$-neighbour by Corollary \ref{cor:3-vertex}-(2), say $w$. Clearly $w$ is a senior vertex, and it follows that $v$ receives 1  from $w$ by R2.
Therefore, we have  $\mu^*(v)\geq 0$ by applying R1. \medskip

\textbf{(3).} Let $k=4$. Observe that if $v$ has no $2$-neighbour, then $\mu^*(v)\geq 0$ by R1 and R3.   Also, if $v$ has only one $2$-neighbour, then $v$ is not adjacent to three $3(1)$-neighbours by Proposition \ref{prop:v-is-not-adj-to-three-3(1)}, and so $\mu^*(v)\geq 0$ by R1 and R3. 
Recall that  all neighbours of $v$ cannot be $2$-vertices by Corollary \ref{cor:4-vertex}-(1). We may then assume that $v$ has either two or three $2$-neighbours, i.e., $2\leq n_2(v)\leq 3$. 
Notice also that if $v$ has a $6^+$-neighbour, then $\mu^*(v)\geq 0$ by R1, R3 and R5. We may therefore assume that $v$ has no $6^+$-neighbour.

\textbf{(3.1).} Let $n_2(v)=3$. Denote by $u_1,u_2,u_3$ the $2$-neighbours of $v$ in a clockwise ordering on the plane. Then the last neighbour of $v$ is  a $4^+$-vertex by Corollary \ref{cor:4-vertex}-(3), say $u_4$. Let $v_i$ be the other neighbour of $u_i$  for $i\in [3]$. By our assumption, $u_4$ is a $4$ or $5$-vertex. This implies that $v$ is a special vertex of type II. It follows from Observation \ref{obs:special-to-bad} that $v$ is a bad $4$-vertex. Then each $v_i$, for $i=1,2,3$,  is a $6$-vertex by Lemma \ref{lem:delta6-special-not-weak-adj}. Let $f_1$ (resp. $f_2$) be the face incident to $v_1,u_1,v,u_2,v_2$ (resp. $v_2,u_2,v,u_3,v_3$). 
Suppose first that both $f_1$ and $f_2$ are $6$-faces. Then the common neighbour of $v_1,v_2$ (resp. $v_2,v_3$) is $3^+$-vertex and it is not $3(1)$-vertex by Lemma \ref{lem:partial-face-4-vertex-deg-2}. It follows from applying R6 that $v_1$ gives at least $\frac{1}{6}$  to $f_1$; $v_2$ gives at least $\frac{1}{3}$  to each $f_1,f_2$; $v_3$ gives at least $\frac{1}{6}$  to $f_2$. Consequently,  $v$ receives totally  $1$  from $f_1,f_2$ by R7 together with Lemma \ref{lem:delta6-f-is-not-inc-to-bad-and-semibad}, and so  $\mu^*(v)\geq 0$ after $v$ transfers $1$  to its each $2$-neighbour by R1. Suppose now that one of $f_1,f_2$ is a $7^+$ face, say $f_1$.  Since $v_1$ and $v_2$ are $6$-vertices, $f_1$ is not incident to any bad or semi-bad vertex except for $v$ when $\ell(f_1)=7$. Thus $f_1$ transfers all its positive charge to $v$ when $\ell(f_1)=7$, and so $v$ receives at least $1$  from $f_1$ by R7. If $\ell(f_1)>7$, then  $\mu(f_1)=\ell(f_1)-6 \geq 2$, and so $f_1$ gives at least $1$  to $v$ by R7 together with Lemma \ref{lem:7-face-lf-3/2}. Consequently, $\mu^*(v)\geq 0$ after $v$ transfers $1$  to its each $2$-neighbour by R1. \medskip

\textbf{(3.2).} Let $n_2(v)=2$. Let $u_1,u_2$ be $2$-neighbours of $v$, and let $v_i$ be the other neighbour of $u_i$ for $i=1,2$. If $v$ has no $3(1)$-neighbour, then $\mu^*(v)\geq 0$ by R1. So we assume that $v$ has at least one $3(1)$-neighbour. In fact, $v$ cannot have more than one $3(1)$-neighbour by Corollary \ref{cor:4-vertex}-(2). Thus, $v$ has a unique $3(1)$-neighbour, say $u_3$.  In addition, the last neighbour of $v$ is a $4^+$-vertex by Corollary \ref{cor:4-vertex}-(2), say $u_4$. By our assumption, $u_4$ is a $4$ or $5$-vertex. Let $N(u_3)=\{v,z,v_3\}$ such that $d(z)=2$ and  $d(v_3)\geq  5$ by Corollary \ref{cor:3-vertex}-(2). Obviously, $v$ is a special vertex of type I, and so both $v_1$ and $v_2$ are $6$-vertices by Lemma \ref{lem:delta6-special-not-weak-adj}. 
Also, we may assume that $v_3$ is neither $6^+$-vertex nor senior $5$-vertex, since otherwise the vertex $u_3$ receives $1$  from $v_3$ by R2 and R5, and so $v$ does not need to send charge to $u_3$ by R3. So, $\mu^*(v)\geq 0$ after $v$ transfers $1$  to its each $2$-neighbour by R1. This forces that $v_3$ is a $5$-vertex with $5$ poor neighbours. In particular, $v$ is a semi-bad $4$-vertex by Observation \ref{obs:special-to-bad}. 

We first assume that $v_1,v_2$ are consecutive vertices on the plane. Let $f_1$ be the face incident to $v_1,u_1,v,u_2,v_2$. If $f_1$ is a $7$-face, then $f_1$ is not incident to any bad or semi-bad vertex except for $v$, since both $v_1$ and $v_2$ are $6$-vertices. It follows that $f_1$ gives $1$  to $v$ by R7. Moreover, if $\ell(f_1)>7$, then  $\mu(f_1)=\ell(f_1)-6 \geq 2$, and so $f_1$ gives at least $1$  to $v$ by R7 together with Lemma \ref{lem:7-face-lf-3/2}. Thus, $\mu^*(v)> 0$ after $v$ transfers $1$  to its each $2$-neighbour by R1. We may further assume that $f_1$ is a $6$-face. Let $f_2$ be the face incident to $v_2,u_2,v,u_3$ and let $f_3$ be the face incident to $v_1,u_1,v,u_4$. 
Notice that $f_1$ is not incident to any bad or semi-bad vertex, except for $v$, by Lemma \ref{lem:delta6-f-is-not-inc-to-bad-and-semibad}. Moreover, each of $f_2,f_3$ has at most $\floor[\Big]{ \frac{\ell(f)}{2}}$ bad and semi-bad vertices by Corollary \ref{cor:at-most-lf/2-bad-and-semi-bad}. In particular, if $\ell(f_i)=6$ for $i\in \{2,3\}$, then $f_i$  has at most two bad or semi-bad vertices by Lemma \ref{lem:6-face-at-most-3-bad-and-semi-bad}.
Once we apply R6, $v_1$ gives $\frac{1}{6}$  to each of $f_1,f_3$; $v_2$ gives $\frac{1}{6}$  to each of $f_1,f_2$. It then follows from applying R7 together with Lemma \ref{lem:delta6-f-is-not-inc-to-bad-and-semibad}, we say that  $f_1$ transfers $\frac{1}{6}+\frac{1}{6}$  to $v$; $f_2$ transfers at least $\frac{1}{12}$  to $v$;  $f_3$ transfers at least $\frac{1}{12}$  to $v$. Consequently, $v$ gets at least $\frac{1}{2}$  from its incident faces. Thus,  we have $\mu^*(v)\geq 0$, after $v$ transfers $\frac{1}{2}$  to $u_3$ by R3, and $1$  to its each $2$-neighbour by R1.   

Now we assume that $v_1,v_2$ are not consecutive vertices on the plane. By symmetry, we may suppose that $v_1,v_3,z,v_2$ are consecutive vertices on the plane.  Let $f_1$ be the face incident to $v_1,u_1,v,u_3,v_3$. Notice that neither $v_1$ nor its neighbours can be bad or semi-bad vertices because of $d(v_1)=6$. Therefore, if $f_1$ is a $7^+$-face, then $f_1$ can have at most $\ceil[\Big]{ \frac{\ell(f)-3}{2}}$ bad or semi-bad vertices by Lemma \ref{lem:bad-vertices-not-adjacent}. So $f_1$ sends at least $\frac{1}{2}$  to $v$ by R7.  Consequently we have $\mu^*(v)\geq 0$, after $v$ transfers $\frac{1}{2}$  to $u_3$ by R3, and $1$  to its each $2$-neighbour by R1.   
We further suppose that $f_1$ is a $6$-face. By Lemma \ref{lem:partial-face-4-vertex-deg-3(1)},  the common neighbour of $v_1,v_3$ is a $3^+$-vertex other than $3(1)$-vertex. Therefore $v_3$ has at most $4$ poor neighbours, however it is a contradiction with the fact that $v_3$ is a $5$-vertex with $5$ poor neighbours. \medskip



\textbf{(4)} Let $k=5$. Observe that if $v$ has at most three $2$-neighbours, then $\mu^*(v)\geq 0$  by applying R1-R4. On the other hand, if $v$ has four $2$-neighbours and one $3^+$-neighbour different from $3(1)$-vertex, then $\mu^*(v)\geq 0$ by applying R1. So we have two cases; $v$ is either a $5(4)$-vertex having a $3(1)$-neighbour or a $5(5)$-vertex, i.e, $v$ is a special vertex of type III or IV. By Lemma \ref{lem:delta6-special-not-weak-adj}-(2),  every vertex which is weak-adjacent to $v$ is a $5^+$-vertex.  Notice also that if $v$ is star-adjacent to a $6^+$-vertex, then the $6^+$-vertex gives $1$  to the $3(1)$-neighbour of $v$ by R5. By this way, the $3(1)$-neighbour of $v$ does not need charge from $v$, and so $\mu^*(v)\geq 0$ after $v$ transfers $1$  to its each $2$-neighbour by R1. Thus we assume that $v$ is  star-adjacent to a $4$-vertex (or $5$-vertex) by Lemma \ref{lem:3-vertex}-(1). Consequently, $v$ is either a bad or a semi-bad $5$-vertex. \medskip

\textbf{(4.1)} Suppose that $v$ is a bad $5$-vertex. Denote by $u_1,u_2,\ldots,u_5$ the $2$-neighbours of $v$ in a clockwise ordering on the plane, and let $v_i$ be the other neighbour of $u_i$ for $i\in [5]$ (see Figure \ref{fig:bad5-a}). Let $f_i$ be the face incident to $v_i,u_i,v,u_{i+1},v_{i+1}$ for each $i\in [5]$ in modulo $5$. Let $F$ be the set of faces incident to $v$, and suppose that $v$ has $t$ \ $7^+$-faces, and $5-t$ \ $6$-faces. Notice that if $f_i\in F$ is a $7^+$-face, then $f_i$ gives at least $\frac{1}{2}$  to $v$ by R7 together with Lemma \ref{lem:7-face-lf-3/2}. Therefore, if $t\geq 2$, then $v$ gets at least $1$  from its incident faces, and so $\mu^*(v)\geq 0$ after $v$ transfers $1$  to its each $2$-neighbour by R1. We further suppose that $t\leq 1$.  Denote by $s_i$ the common neighbour of $v_i$ and $v_{i+1}$ for $i\in [5]$ in modulo $5$ if $f_i$ is a $6$-face. It is clear that none of $s_i$'s can be $4(3)$-vertex or $4(2)$-vertex adjacent to a $3(1)$-vertex. \medskip

\textbf{(4.1.1)} Let $t=0$. Then, all faces incident to $v$ are $6$-faces. Let $f_i=v_iu_ivu_{i+1}v_{i+1}s_i$ for each $i\in [5]$ in modulo $5$ (see Figure \ref{fig:Fbad5-a}).  By Lemma \ref{lem:partial-face-5-vertex-no-poor}, at most two of $s_i$'s can be poor vertices, and they are not in consecutive faces. We first suppose that at most one of $s_i$ is a poor vertex, say $s_5$. Then each of $v_2,v_3,v_4$  has at least $1$ positive charge after applying R1-R5, since each of $v_2,v_3,v_4$ has at least two non-poor vertices. Recall that each $v_i$ is a $5^+$-vertex. Once we apply R4-R6, each $v_i$, for $i=2,3,4$, gives at least $\frac{1}{5}$ to each incident face. Since each $f_i$ has only one bad or semi-bad vertex by Lemma \ref{lem:delta6-f-is-not-inc-to-bad-and-semibad}, $v$ receives $\frac{1}{5}+\frac{2}{5}+\frac{2}{5}+\frac{1}{5}$ from $f_1,f_2,f_3,f_4$, respectively. Then, $v$ gets totally at least $\frac{6}{5}$  from its incident faces. After $v$ transfers $1$  to each $2$-neighbour by R1, we have $\mu^*(v)\geq \frac{1}{5}$. 

We now suppose that two of $s_i$'s are poor vertices, say $s_1,s_3$. It follows from Corollary \ref{cor:delta6-no-consecutive} that  $v_1,v_2,v_3,v_4$ are $6$-vertices. Since $s_2,s_4,s_5$ are not poor vertices,  each $v_i$, for $i\in [5]$, has $1$ positive charge after applying R1-R3.  Once we apply R6, each $v_i$, for $i=1,2,3,4$, gives at least $\frac{1}{6}$  to each incident face, and $v_5$ gives at least $\frac{1}{5}$  to each incident face. Since each $f_i$ has only one special vertex  by Lemma \ref{lem:delta6-f-is-not-inc-to-bad-and-semibad}, $v$ receives $3\times \frac{2}{6}+2\times\frac{11}{30}$ from $f_1,f_2,\ldots,f_5$, respectively. Consequently, we have $\mu^*(v)>0$ after applying R1. \medskip

\textbf{(4.1.2).} Let $t=1$. Then there exist consecutive four $6$-faces, say $f_1,f_2,f_3,f_4$. So $f_5$ is a $7^+$-face, and it gives at least $\frac{1}{2}$  to $v$ by R7 together with Lemma \ref{lem:7-face-lf-3/2}. Let $f_i=v_iu_ivu_{i+1}v_{i+1}s_i$ for each $i\in [4]$.  Remark that each $f_i$ for $i\in [4]$ has only one bad or semi-bad vertex by Lemma \ref{lem:delta6-f-is-not-inc-to-bad-and-semibad}. On the other hand, at most two of $s_i$'s can be poor vertices for $i\in [4]$, and they do not belong to consecutive faces by Corollary \ref{cor:delta6-no-consecutive}-(1). We first suppose that at most one of $s_i$'s is a poor vertex. By symmetry we consider only the cases that $s_3$ or $s_4$ is a poor vertex. If $s_3$ is a poor vertex, then $v_3,v_4$ are $6$-vertices by Corollary \ref{cor:delta6-no-consecutive}-(2). Since $s_1$ and $s_2$ are not poor vertices, $v_2$ has at least $1$ positive charge after applying R1-R3, R5. Also, each of $v_3,v_4$ has at least $1$ positive charge after applying R5.  Once we apply R4 and R6, $v_2$  gives at least $\frac{1}{5}$  to every incident face, and each of $v_3,v_4$  gives at least $\frac{1}{6}$  to every incident face.  Since each $f_i$ for $i\in [4]$ has only one bad or semi-bad vertex, $v$ receives at least $\frac{1}{5}+\frac{11}{30}+\frac{2}{6}+\frac{1}{6}$  from $f_1,f_2,f_3,f_4$ by R7, respectively. Thus, $\mu^*(v)>0$ after $v$ gives $1$  to its each $2$-neighbour by R1. Now we suppose that $s_4$ is a poor vertex. Similarly as above, then $v_4,v_5$ are $6$-vertices, and  each of $v_2,v_3$ has at least $1$ positive charge after applying R1-R3, R5. Also, $v_4$ has at least $1$ positive charge after applying R5. It follows that each of $v_2,v_3$ gives at least $\frac{1}{5}$  to each incident face. Since each $f_i$ for $i\in [4]$ has only one bad or semi-bad vertex, $v$ receives at least $\frac{1}{5}+\frac{2}{5}+\frac{11}{30}+\frac{1}{6}$ from $f_1,f_2,f_3,f_4$, respectively. Consequently, we have $\mu^*(v)>0$.

We now suppose that two of $s_i$'s are poor vertices, say $s_1,s_3$. It follows from Corollary \ref{cor:delta6-no-consecutive} that $v_1,v_2,v_3,v_4$ are $6$-vertices. Since $s_2,s_4$ are not poor vertices, each $v_i$ for $i\in \{2,3,4\}$ has at least $1$ positive charge after applying R5.  
Once we apply R6, each $v_i$, for $i\in \{2,3,4\}$, gives at least $\frac{1}{6}$  to each incident face. Since each $f_i$ for $i\in [4]$ has only one bad or semi-bad vertex, $v$ receives $ \frac{1}{6}+\frac{2}{6}+ \frac{2}{6}+ \frac{1}{6}$ from $f_1,f_2,f_3,f_4$, respectively. Thus, we have $\mu^*(v)\geq 0$. \medskip

\textbf{(4.2).} Suppose now that $v$ is a semi-bad $5$-vertex. Denote by $u_1,u_2,u_3,u_4$ the $2$-neighbours of $v$ in a clockwise ordering on the plane, and and let $v_i$ be the other neighbour of $u_i$ for $i\in [4]$. Let $u_5$ be the $3(1)$-neighbour of $v$ with $N(u_5)=\{v,z,v_5\}$ such that  $d(z)=2$ and $d(v_5)\geq 4$ (see Figure \ref{fig:bad5-b}). By Lemma \ref{lem:delta6-special-not-weak-adj}-(2),  $v$ is not weak-adjacent to any $4^-$-vertex. So each $v_i$, for $i\in [4]$, is a $5^+$-vertex. 
Let $F$ be the set of faces incident to $v$, and let $f_i\in F$ be incident to $u_i,v,u_{i+1}$ for $i\in [3]$. If $\ell(f_i)\geq 7$, then $f_i$ has at most $\ceil[\Big]{ \frac{\ell(f)-3}{2}}$ bad or semi-bad vertices by Lemma \ref{lem:7-face-lf-3/2}. 
This infer that if one of $f_1,f_2,f_3$ is a $7^+$-face, then it gives at least $\frac{1}{2}$  to $v$ by R7, and so we have $\mu^*(v)\geq 0$ after $v$ transfers $1$  to each $2$-neighbour by R1 and $\frac{1}{2}$  to the $3(1)$-neighbour by R3. We further suppose that $f_1,f_2,f_3$ are $6$-faces.   Let $f_i=v_iu_ivu_{i+1}v_{i+1}s_i$ for each $i\in [3]$ (see Figure \ref{fig:Fbad5-b}). By Lemma \ref{lem:delta6-f-is-not-inc-to-bad-and-semibad}, each of $f_1,f_2,f_3$ has only one bad or semi-bad vertex.  Recall also that none of $s_i$'s can be $4(3)$-vertex or $4(2)$-vertex adjacent to a $3(1)$-vertex.  Let us divide the rest of the proof into three cases as follows.

If none of $s_i$, for $i\in [3]$, is a poor vertex, then each of $v_2,v_3$ gives at least $\frac{1}{5}$  to each incident faces by R4-R5. Since each of $f_1,f_2,f_3$ has only one bad or semi-bad vertex, $v$ receives $\frac{1}{5}+\frac{2}{5}+\frac{1}{5}$ from $f_1,f_2,f_3$, respectively. Then, $v$ gets totally at least $\frac{4}{5}$  from its incident faces. After $v$ transfers $1$  to its each $2$-neighbour by R1 and $\frac{1}{2}$  to the $3(1)$-neighbour by R3, we have $\mu^*(v)\geq \frac{3}{10}$. 

Suppose next that only one of $s_1,s_2,s_3$ is a poor vertex. If $s_1$ is a poor vertex, then $v_2$ is a $6$-vertex by Corollary \ref{cor:delta6-no-consecutive}-(2), and so it gives at least $\frac{1}{6}$  to its each incident faces by R6. On the other hand, $v_3$ gives at least $\frac{1}{5}$  to each incident faces by R4.  Since each $f_1,f_2,f_3$ has only one bad or semi-bad vertex, $v$ receives at least $\frac{1}{6}+\frac{11}{30}+\frac{1}{5}$ from $f_1,f_2,f_3$, respectively. Then, $v$ gets totally at least $\frac{11}{15}$  from its incident faces. Thus, we have $\mu^*(v)\geq \frac{7}{30}$ after $v$ transfers $1$  to its each $2$-neighbour by R1 and $\frac{1}{2}$  to the $3(1)$-neighbour by R3.

Finally suppose that two of  $s_1,s_2,s_3$ are  poor vertices. Clearly, those vertices should be $s_1,s_3$ by Corollary \ref{cor:delta6-no-consecutive}-(1). It follows that each $v_i$, for $i\in [4]$, is a $6$-vertex, so each of them gives at least $\frac{1}{6}$  to its each incident faces. Then, $v$ gets totally at least $1$  from $f_1,f_2,f_3$. Similarly as above, we have $\mu^*(v)>0$.\medskip

\textbf{(5).} Let $k\geq 6 $. By R5, $v$ gives $1$  to its some neighbours, and afterwards it transfers its positive charge equally to its incident faces by R6.  Consequently, $\mu^*(v)\geq 0$.

%
%
%
%
%
%
%
%



\begin{thebibliography}{99}

%

\bibitem{bu-zu}
Y. Bu, X. Zhu, An optimal square coloring of planar graphs, Journal of Combinatorial Optimization, 24 (2012) 580–592.


\bibitem{dong}
W. Dong,  and W. Lin. An improved bound on 2-distance coloring plane graphs with girth 5. Journal of Combinatorial Optimization, 32(2), 645-655, (2016).

\bibitem{hartke}
S. G. Hartke,  S. Jahanbekam, and B. Thomas.  The chromatic number of the square of subcubic planar graphs. arXiv preprint arXiv:1604.06504, (2016). 

\bibitem{molloy}
M. Molloy and M. R. Salavatipour, A bound on the chromatic number of the square
of a planar graph, Journal of Combinatorial Theory, Series B, 94:189–213, 2005.

\bibitem{thomassen}
C. Thomassen. The square of a planar cubic graph is 7-colorable. Journal of Combinatorial Theory, Series B, 128:192–218, 2018.


\bibitem{van-den}
J. van den Heuvel, S. McGuinness, Coloring of the square of planar graph, J. Graph Theory 42 (2003) 110–124.

\bibitem{wegner}
G. Wegner. Graphs with given diameter and a coloring problem. Technical report, University of Dormund,
1977.

\bibitem{west}
D. B. West. Introduction to graph theory, volume 2. Prentice hall Upper Saddle River, 2001.


\bibitem{zhu18} J. Zhu and Y. Bu, Minimum 2-distance coloring of planar graphs and channel assignment. Journal of Combinatorial Optimization, (2018) 36(1), 55-64.

\bibitem{zhu22} J. Zhu, Y., Bu and H. Zhu, Wegner's Conjecture on 2-distance coloring for planar graphs. Theoretical Computer Science,  926, 71-74, (2022).


%


\end{thebibliography}
\end{document}